\documentclass[draft,journal,onecolumn,12pt]{IEEEtran}

\usepackage{showtags} 

\usepackage{graphicx,courier,amssymb,amsmath,times}
\usepackage{epsfig,amsmath}
\usepackage{bm,graphicx,amssymb,psfrag}
\usepackage{amsfonts}
\usepackage[latin1]{inputenc}

\usepackage{booktabs}
 \usepackage{amsthm}
 
\usepackage{algorithmic}
\usepackage{algorithm}

\algsetup{linenosize=\scriptsize}

\usepackage{cite}

\usepackage{hyperref}

\usepackage[usenames]{color}

\usepackage{currfile} 

\usepackage[nolist]{acronym}

\newcommand{\W}{\mathcal{TW}}

\def\nmin{n}
\def\nmax{N}

\def\reggamma{P}

\newcommand{\mb}[1] {{\mathbf{#1}}}

\newcommand{\td}[1] {\tilde{#1}}

\def\Nmin{n}
\def\Nmax{N}

\def\Hypergeometric1F1{${}_{1}F_{1}$}

\def\sgn{\text{sgn}}
\def\erf{\text{erf}}
\def\erfc{\text{erfc}}
\def\aappx{\alpha}
\def\kappx{{k}}
\def\tappx{\theta}

\def\Fj{\reggamma}

\def\x{x}

\def\betareg{{\mathcal{B}}}

\def\p{\mathsf{p}}
\def\n{\mathsf{n}}
\def\m{\mathsf{m}}

\newcommand{\mat}[1] {{{\mathbf{#1}}}}
\newcommand{\vect}[1] {{{\mathbf{#1}}}}

\newtheorem{theorem}{Theorem}
\newtheorem{lemma}{Lemma}

\renewcommand*{\today}{Submitted 2015}

\begin{document}

\begin{acronym}
\scriptsize
\acro{AcR}{autocorrelation receiver}
\acro{ACF}{autocorrelation function}
\acro{ADC}{analog-to-digital converter}
\acro{AWGN}{additive white Gaussian noise}
\acro{BCH}{Bose Chaudhuri Hocquenghem}
\acro{BEP}{bit error probability}
\acro{BFC}{block fading channel}
\acro{BPAM}{binary pulse amplitude modulation}
\acro{BPPM}{binary pulse position modulation}
\acro{BPSK}{binary phase shift keying}
\acro{BPZF}{bandpass zonal filter}
\acro{CD}{cooperative diversity}
\acro{CDF}{cumulative distribution function}
\acro{CCDF}{complementary cumulative distribution function}
\acro{CDMA}{code division multiple access}
\acro{c.d.f.}{cumulative distribution function}
\acro{ch.f.}{characteristic function}
\acro{CIR}{channel impulse response}
\acro{CR}{cognitive radio}
\acro{CSI}{channel state information}
\acro{DAA}{detect and avoid}
\acro{DAB}{digital audio broadcasting}
\acro{DS}{direct sequence}
\acro{DS-SS}{direct-sequence spread-spectrum}
\acro{DTR}{differential transmitted-reference}
\acro{DVB-T}{digital video broadcasting\,--\,terrestrial}
\acro{DVB-H}{digital video broadcasting\,--\,handheld}
\acro{ECC}{European Community Commission}
\acro{ELP}{equivalent low-pass}
\acro{FCC}{Federal Communications Commission}
\acro{FEC}{forward error correction}
\acro{FFT}{fast Fourier transform}
\acro{FH}{frequency-hopping}
\acro{FH-SS}{frequency-hopping spread-spectrum}
\acro{GA}{Gaussian approximation}
\acro{GPS}{Global Positioning System}
\acro{HAP}{high altitude platform}
\acro{i.i.d.}{independent, identically distributed}
\acro{IFFT}{inverse fast Fourier transform}
\acro{IR}{impulse radio}
\acro{ISI}{intersymbol interference}
\acro{LEO}{low earth orbit}
\acro{LOS}{line-of-sight}
\acro{BSC}{binary symmetric channel}
\acro{MB}{multiband}
\acro{MC}{multicarrier}
\acro{MF}{matched filter}
\acro{m.g.f.}{moment generating function}
\acro{MI}{mutual information}
\acro{MIMO}{multiple-input multiple-output}
\acro{MISO}{multiple-input single-output}
\acro{MRC}{maximal ratio combiner}
\acro{MMSE}{minimum mean-square error}
\acro{M-QAM}{$M$-ary quadrature amplitude modulation}
\acro{M-PSK}{$M$-ary phase shift keying}
\acro{MUI}{multi-user interference}
\acro{NB}{narrowband}
\acro{NBI}{narrowband interference}
\acro{NLOS}{non-line-of-sight}
\acro{NTIA}{National Telecommunications and Information Administration}
\acro{OC}{optimum combining}
\acro{OFDM}{orthogonal frequency-division multiplexing}
\acro{p.d.f.}{probability density function}
\acro{PAM}{pulse amplitude modulation}
\acro{PAR}{peak-to-average ratio}
\acro{PDP}{power dispersion profile}
\acro{p.m.f.}{probability mass function}
\acro{PN}{pseudo-noise}
\acro{PPM}{pulse position modulation}
\acro{PRake}{Partial Rake}
\acro{PSD}{power spectral density}
\acro{PSK}{phase shift keying}
\acro{QAM}{quadrature amplitude modulation}
\acro{QPSK}{quadrature phase shift keying}
\acro{r.v.}{random variable}
\acro{R.V.}{random vector}
\acro{SEP}{symbol error probability}
\acro{SIMO}{single-input multiple-output}
\acro{SIR}{signal-to-interference ratio}
\acro{SISO}{single-input single-output}
\acro{SINR}{signal-to-interference plus noise ratio}
\acro{SNR}{signal-to-noise ratio}
\acro{SS}{spread spectrum}
\acro{TH}{time-hopping}
\acro{ToA}{time-of-arrival}
\acro{TR}{transmitted-reference}
\acro{UAV}{unmanned aerial vehicle}
\acro{UWB}{ultrawide band}
\acro{UWBcap}[UWB]{Ultrawide band}
\acro{WLAN}{wireless local area network}
\acro{WMAN}{wireless metropolitan area network}
\acro{WPAN}{wireless personal area network}
\acro{WSN}{wireless sensor network}
\acro{WSS}{wide-sense stationary}

\acro{SW}{sync word}
\acro{FS}{frame synchronization}
\acro{BSC}{binary symmetric channels}
\acro{LRT}{likelihood ratio test}
\acro{GLRT}{generalized likelihood ratio test}
\acro{LLRT}{log-likelihood ratio test}
\acro{$P_{EM}$}{probability of emulation, or false alarm}
\acro{$P_{MD}$}{probability of missed detection}
\acro{ROC}{receiver operating characteristic}
\acro{AUB}{asymptotic union bound}
\acro{RDL}{"random data limit"}
\acro{PSEP}{pairwise synchronization error probability}

\acro{SCM}{sample covariance matrix}

\acro{PCA}{principal component analysis}

\end{acronym}

\title{On the probability that all eigenvalues of Gaussian, Wishart, and double Wishart random matrices lie within an interval 
}
%
\author{
    Marco~Chiani,~\IEEEmembership{Fellow,~IEEE}
\\
\thanks{Marco Chiani
        is with 
        DEI,  University of Bologna, 
        V.le Risorgimento 2, 40136 Bologna, Italy
        (e-mail: marco.chiani@unibo.it). 
This work was partially supported by the Italian Ministry of Education, Universities and Research (MIUR) under the PRIN Research Project ``GRETA'', and in part by the H2020 European Project EuroCPS,  Grant 644090. 
}
}

\markboth{Accepted for publication on IEEE Trans. On Inf. Theory}{Prob. all eigenvalues of random matrices lie within an interval}

\date{\today}

\maketitle




\begin{abstract}
We derive the probability that all eigenvalues of a random matrix $\bf M$ lie within an arbitrary interval $[a,b]$, $\psi(a,b)\triangleq\Pr\{a\leq\lambda_{\min}({\bf M}), \lambda_{\max}({\bf M})\leq b\}$, when $\bf M$ is a real or complex finite dimensional Wishart, double Wishart, or Gaussian symmetric/Hermitian matrix. 
 We give efficient recursive formulas allowing the exact evaluation of $\psi(a,b)$ for Wishart matrices, even with large number of variates and degrees of freedom. We also prove that the probability that all eigenvalues are within the limiting spectral support (given by the Mar{\v{c}}enko-Pastur or the semicircle laws) 
tends for large dimensions to the universal values $0.6921$ and $0.9397$ for the real and complex cases, respectively. 
Applications include improved bounds for the probability that a Gaussian measurement matrix has a given restricted isometry constant in compressed sensing. 
\end{abstract}

\begin{IEEEkeywords}
Random Matrix Theory, Principal Component Analysis, Compressed Sensing, eigenvalues distribution, Tracy-Widom distribution, Wishart matrices, Gaussian Orthogonal Ensemble, MANOVA, Jacobi ensemble.
\end{IEEEkeywords}

\section{Introduction}
\label{sec:intro}
\IEEEPARstart{M}{any} mathematical models in information theory and physics are formulated by using matrices with random elements.  
In particular, the distribution of the eigenvalues of Gaussian, Wishart, and double Wishart random matrices plays a key role in multivariate analysis, including principal component analysis, analysis of large data sets, information theory, signal processing and mathematical physics \cite{And:03,Mui:B82,Meh:91,Wint:87,Ede:88,Tel:99,Joh:01,ChiWinZan:J03,CanTao:05,PenGar:09,CheMcK:12}. 

Most of the problems concern the distribution of the smallest and/or of the largest eigenvalue, or of a randomly picked (unordered) eigenvalue. 
For example, in compressed sensing the probability that a randomly generated measurement matrix has a given restricted isometry constant  is related to the probability $\Pr\left\{a \leq \lambda_{\min}({\bf M}), \lambda_{\max}({\bf M}) \leq b \right\}$, where $\lambda_{\min}({\bf M})$, $\lambda_{\max}({\bf M})$ denote the minimum and maximum eigenvalues of the matrix ${\bf M}={\bf X}^\dag {\bf X}$, and ${\bf X}$ is a submatrix of the measurement matrix \cite{CanTao:05,Don:06,Can:08}. 
To mention another example, several stability problems in physics, in complex networks and in complex ecosystems are related to the probability that all eigenvalues of a random symmetric matrix (for instance with Gaussian entries) are negative 
 \cite{May:72,AazEas:06,DeaMaj:08,MarMca:13}.  
This probability is also important in mathematics, as it is related to the expected number of minima in random polynomials \cite{DedMal:07}.

Owing to the difficulties in computing the exact marginal distributions of eigenvalues, asymptotic formulas for matrices with large dimensions are often used as approximations.  
One important example is the Wigner semicircular law, giving the asymptotic distribution of a randomly picked eigenvalue of a symmetric/Hermitian random matrix ${\bf M}$ having 
zero mean \ac{i.i.d.} entries above the diagonal. 
The law applies to a wide range of distributions for the entries \cite{BaiSil:06}, and in particular when ${\bf M}={\bf X} + {\bf X}^\dag$ and the elements of ${\bf X}$ are zero mean (real or complex) \ac{i.i.d.} Gaussian. In this situation, the symmetric matrix ${\bf M}={\bf X} + {\bf X}^\dag$ belongs to the Gaussian Orthogonal Ensemble (GOE) and to the Gaussian Unitary Ensemble (GUE), for the real and complex cases, respectively. 
Another key result is the Mar{\v{c}}enko-Pastur law, giving the asymptotic distribution for one randomly picked eigenvalue of the random matrix ${\bf M}={\bf X}^\dag {\bf X}$. This is  related to the sample covariance matrix, and therefore of primary importance in statistics and signal processing. The limiting Mar{\v{c}}enko-Pastur law applies to a wide class of distributions for the entries of ${\bf X}$, including the case when ${\bf X}$ has \ac{i.i.d.} Gaussian entries, and thus ${\bf M}={\bf X}^\dag {\bf X}$ is a white Wishart matrix  \cite{BaiSil:06}.

The limiting value and the limiting distribution of the extreme eigenvalues (largest or smallest) have also been studied intensively \cite{BaiSil:06,TraWid:09}. The Tracy-Widom laws give the asymptotic distribution of the extreme eigenvalues around the limiting values  \cite{TraWid:94,TraWid:96,Joha:00,Joh:01,Sos:02,Joh:09,TraWid:09,FelSod:10,Ma:12}. Large deviation methods are  used  in \cite{DeaMaj:06,DeaMaj:08,NadMaj:11}  to derive the asymptotic behavior of the  distribution of the  largest eigenvalue far from its mean value. 

For non-asymptotic analysis, where the random matrices have finite dimensions, deriving the distribution of eigenvalues is generally difficult, especially for the real Wishart and GOE.  Recently, the exact distribution of the largest eigenvalue, as well as  efficient recursive methods for its numerical computation, has been  found for real white Wishart, multivariate Beta (also known as double Wishart or MANOVA), and GOE matrices  \cite{Chi:J14,Chi:J16}. 
These matrices are also denominated, using the names of the associated weight polynomials, as Laguerre (Wishart), Jacobi (double Wishart), and Hermite (Gaussian) ensembles.

In  this paper we give new expressions and efficient recursive methods for the evaluation of the function $\psi(a,b)$, defined as the probability 
 that all eigenvalues of a random matrix $\bf M$ are within an arbitrary interval $[a,b]$, when $\bf M$ is a real finite dimensional white Wishart, double Wishart, or Gaussian symmetric matrix. 
For completeness we  provide also the results for complex Wishart (with arbitrary covariance), complex double Wishart, and GUE, by specializing \cite[Th. 7]{ChiZan:C08}. 
The marginal cumulative distribution of the smallest eigenvalue and of the largest eigenvalue can be seen as the particular cases $1-\psi(a,\infty)$ and $\psi(-\infty,b)$, respectively.

We then derive simple and accurate approximations to $\psi(a,b)$ based on the incomplete gamma function and valid for large matrices, and prove 
 that the probability that all eigenvalues are within the limiting spectral support (given by the Mar{\v{c}}enko-Pastur and the semicircle laws) tends for large dimensions to the universal values $0.6921$ and $0.9397$ for the real and complex cases, respectively.

Throughout the paper we indicate with $\Gamma(.)$ the gamma function, with $\gamma\left(a; x, y\right)=\int_{x}^{y} t^{a-1} e^{-t} dt$ the generalized incomplete gamma function, with  $\reggamma(a,x)=\frac{1}{\Gamma(a)}\gamma(a;0,x)$ the regularized lower incomplete gamma function, with $P(a;x,y)=\frac{1}{\Gamma(a)} \int_{x}^{y} t^{a-1} e^{-t} dt=\reggamma(a,y)-\reggamma(a,x)$ the generalized regularized incomplete gamma function, 
 with $\betareg\left(x,y;a,b\right)=\int_{x}^{y} t^{a-1} (1-t)^{b-1} dt$ the  incomplete beta function \cite[Ch. 6]{AbrSte:B70}, with $()^\dag$ transposition and complex conjugation, {and with $|\cdot|$ or $\det(\cdot)$ the determinant}. When possible we use capital letters for random variables, and bold for vectors and matrices. 
We say that a random variable $Z$ has a standard complex Gaussian distribution (denoted ${\mathcal{CN}}(0, 1)$) if $Z = Z_1 + i Z_2$, where $Z_1$ and $Z_2$ are i.i.d. real Gaussian ${\mathcal N}(0, 1/2)$. A complex random vector $\vect{X}$ is Gaussian circularly symmetric if its \ac{p.d.f.} has the form $f(\vect{x}) \propto \exp\left(-\vect{x}^\dag \mat{\Sigma}^{-1} \vect{x}\right)$, where $\mat{\Sigma}$ is the covariance matrix. Note that this implies that $\vect{X}$ is zero mean. When $\mat{\Sigma}=\mat{I}$ the entries of $\vect{X}$ are \ac{i.i.d.} ${\mathcal{CN}}(0, 1)$.
\section{
Real Wishart and Gaussian symmetric matrices}
\label{sec:exactreal}
%
%

\subsection{Real Wishart matrices}

Assume a Gaussian real $\nmin \times \nmax$ matrix ${\bf X}$ with \ac{i.i.d.} columns, each with zero mean and covariance ${\bf \Sigma}$, and $\nmax \geq \nmin$. The  real matrix ${\bf M =X X}^T$ is called Wishart, and its distribution indicated as ${\bf \mathcal{W}}_{\nmin}(\nmax, {\bf \Sigma})$. When ${\bf \Sigma} \propto {\bf I}$ the matrix is called white or uncorrelated Wishart. 

Denoting 
 ${\Gamma}_{m}(a)=\pi^{m (m-1)/4}  \prod_{i=1}^{m}\Gamma(a-(i-1)/2)$, the joint \ac{p.d.f.} of the (real) ordered eigenvalues $\lambda_1 \geq \lambda_2 \cdots \geq
\lambda_{\Nmin} \geq 0$ of the real Wishart matrix ${\bf M} \sim {\bf \mathcal{W}}_{\nmin}(\nmax, {\bf I})$ is \cite{Jam:64, And:03} 
\begin{equation}\label{eq:jpdfuncorrnodet}
f(x_{1}, \ldots, x_{\Nmin}) = K \,
     \prod_{i=1}^{\Nmin}e^{-x_{i}/2}x_{i}^{\alpha}
       \prod_{i<j}^{\Nmin}
    \left(x_{i}-x_{j}\right) 
\end{equation}
where {$x_1 \geq x_2 \geq \cdots \geq x_{\Nmin} \geq 0$,  $\alpha \triangleq (\Nmax-\Nmin-1)/2$}, and $K$ is a normalizing constant given by
\begin{equation*}
K = \frac{\pi^{\Nmin^2 /2}}
        {2^{\Nmin \Nmax /2} \Gamma_{\Nmin}(\Nmax/2)  \Gamma_{\Nmin}(\Nmin/2)} \, .
\label{eq:K}
\end{equation*}
%
%
\subsection{Real symmetric Gaussian matrices (GOE)}
The Gaussian Orthogonal Ensemble is constituted by the real $n \times n$ symmetric matrices whose entries are \ac{i.i.d.} Gaussian  ${\mathcal{N}}(0, 1/2)$ on the upper-triangle, and  \ac{i.i.d.} ${\mathcal{N}}(0, 1)$ on the diagonal \cite{TraWid:09}. 
The joint \ac{p.d.f.} of the eigenvalues for GOE is \cite{Meh:91,TraWid:09}
\begin{equation} \label{eq:jpdfuncorrGOE}
f(x_{1}, \ldots, x_{n}) = K_{GOE} 
     \prod_{i=1}^{n}e^{-x_{i}^2/2} \prod_{i<j}^{n}
    \left(x_{i}-x_{j}\right) 
\end{equation}
where {$x_1 \geq x_2 \cdots \geq x_n$} and the normalizing constant is $K_{GOE}=[2^{n/2} \prod_{\ell=1}^{n} \Gamma(\ell/2)]^{-1}$. Note that here the eigenvalues are distributed over all the reals.


\subsection{Real multivariate beta (double Wishart) matrices}

Let ${\bf X, Y} $ denote two independent real Gaussian $\p \times \m$ and $\p \times \n$ matrices with $\m, \n \geq \p$, each constituted by zero mean \ac{i.i.d.} columns with common covariance.
 Multivariate analysis of variance (MANOVA) is based on the statistic of the eigenvalues of ${\bf{(A+B)}}^{-1} \bf{B}$ (beta matrix), where ${\bf A =X X}^T$ and ${\bf B =Y Y}^T$ are independent Wishart matrices.  These eigenvalues are clearly related to the eigenvalues of ${\bf{A}}^{-1} \bf{B}$ (double Wishart or multivariate beta).
 
The joint distribution  of $s$ non-null eigenvalues of a multivariate real beta matrix in the null case can be written in the form \cite[page 112]{Mui:B82}, \cite[page 331]{And:03}, 
\begin{equation}\label{eq:jpdfuncorrnodetMB}
f(\x_{1}, \ldots, \x_{s}) = K_{MB} \,
     \prod_{i=1}^{s} \x_{i}^{m} (1-\x_{i})^n
       \prod_{i<j}^{s}
    \left(\x_{i}-\x_{j}\right) 
\end{equation}
where $1 \geq \x_{1} \cdots \geq \x_{s} \geq 0$, and $K_{MB}$ is a normalizing constant given by
\begin{equation*}
K_{MB} = \pi ^{s/2} \prod _{i=1}^s  \frac{ \Gamma \left(\frac{i+2 m+2 n+s+2}{2} \right)}{ \Gamma
   \left(\frac{i}{2}\right) \Gamma \left(\frac{i+2 m+1}{2} \right) \Gamma \left(\frac{i+2 n+1}{2} \right)}\, .
\label{eq:C}
\end{equation*}
With the notation introduced above, this is the distribution of the eigenvalues of ${\bf (A+B)^{-1}B}$ with parameters 
%
$s=\p, \quad m=(\n-\p-1)/2 , \quad  n=(\m-\p-1)/2 \,.$
%
The marginal distribution of the largest eigenvalue is of basic importance in testing hypotheses and constructing confidence regions in MANOVA according to the Roy's largest root criterion 
 \cite[page 333]{And:03}, \cite{Chi:J16}. 

\subsection{The function $\psi(a,b)$ for real Wishart matrices}
The following is a new theorem for real white Wishart matrices.

\begin{theorem}
\label{th:cdfwishartreal}
The probability that all non-zero eigenvalues of the real Wishart matrix ${\bf M} \sim {\bf \mathcal{W}}_{\nmin}(\nmax, {\bf I})$ are within the interval $[a,b] \subset [0,\infty)$ is
\begin{equation}
\label{eq:cdfwishartreal}
\psi(a,b)=
 K' \,   \sqrt{\left|{\bf A}(a,b)\right|}
\end{equation}
with the constant
$$
K' = K \, 2^{\alpha \nmin+\nmin (\nmin+1)/2} \prod_{\ell=1}^{\nmin} \Gamma\left(\alpha+\ell\right) \,.
$$
%
In \eqref{eq:cdfwishartreal}, when $\nmin$ is even the elements of the $\nmin \times \nmin$ skew-symmetric matrix ${\bf A}(a,b)$  are\footnote{Note that skew symmetry implies $a_{i,j}=-a_{j,i}$ and $a_{i,i}=0$.}
\begin{multline}
\label{eq:aij}
a_{i,j} =  
\left[\Fj\left(\alpha_j,{b}/{2}\right)+\Fj\left(\alpha_j,{a}/{2}\right)\right] P\left(\alpha_i;{a}/{2},{b}/{2}\right) 
- \frac{2}{\Gamma(\alpha_i)} \int_{a/2}^{b/2} x^{\alpha+i-1} e^{-x} \Fj(\alpha_j, x) dx 
\end{multline}
for $i,j=1,\ldots,\nmin$, where $\alpha_\ell=\alpha+\ell=(\Nmax-\Nmin-1)/2+\ell$.

When $\nmin$ is odd, the elements of the $(\nmin+1) \times (\nmin+1)$ skew-symmetric matrix ${\bf A}(a,b)$ are as in \eqref{eq:aij}, with the additional elements
\begin{eqnarray}  \label{eq:aijoddwishartreal}
  a_{i,\nmin+1}&=&P\left(\alpha_i;{a}/{2},{b}/{2}\right)   \qquad i=1, \ldots, \nmin  \nonumber \\
a_{\nmin+1,j}&=&-a_{j,\nmin+1} \qquad\qquad j=1, \ldots, \nmin \\   
a_{\nmin+1,\nmin+1}&=&0  \, . \nonumber
\end{eqnarray}
%

Moreover, the elements $a_{i,j}$ can be computed iteratively, without numerical integration or series expansion, starting from $a_{i,i}=0$ with the iteration
\begin{multline}\label{eq:aij+1}
a_{i,j+1} =  a_{i,j}+\frac{\Gamma(\alpha_i+\alpha_j)}{ \Gamma(\alpha_j+1) \Gamma(\alpha_i) 2^{\alpha_i+\alpha_j-1}}{P}(\alpha_i+\alpha_j;a,b)  \\
 - \frac{g(\alpha_j,a/2)+g(\alpha_j,b/2)}{\Gamma(\alpha_j+1)}{P}\left(\alpha_i;{a}/{2},{b}/{2}\right)
\end{multline}
for $j=i, \ldots, \nmin-1$, with 
 $g(a,x)=x^a e^{-x}$. 
\end{theorem}

\begin{proof}
We have to integrate the \ac{p.d.f.} in \eqref{eq:jpdfuncorrnodet}. First we observe that an identity related to Vandermonde matrices gives $\det\left[ \left\{y_i^{j-1} \right\} \right]=\prod_{i<j}(y_j-y_i)$. 
Then, for arbitrary constants $\gamma_\ell \neq 0$  we write
\begin{align}
\underset{{a \, \leq x_{1} < \cdots  < x_{\Nmin} \leq \, b}}{\idotsint} & f (x_{\Nmin}, \ldots, x_{1}) d{\bf x} = \nonumber \\
= K' & \underset{{a/2 \, \leq y_1 < \cdots < y_{\Nmin} \leq \, b/2}}{\idotsint}   \prod_{i=1}^{\Nmin} \gamma_i \, y_i^{\alpha}e^{-y_i}  \prod_{i<j}^{\Nmin} (y_j-y_i)  d{\bf y} \nonumber \\
     = K' & \underset{{a/2 \, \leq y_1 < \cdots < y_{\Nmin} \leq \, b/2}}{\idotsint}  \det\left[\left\{ \Phi_i( y_j)\right\}\right]  d{\bf y} \label{eq:th1proof}
\end{align}
with $\Phi_i(y)=\gamma_i \, y^{\alpha+i-1} e^{-y}$ and $K'= K 2^{\alpha \nmin +\nmin(\nmin+1)/2} \prod_{\ell=1}^{\Nmin} \gamma_\ell^{-1}$.
To evaluate this integral we recall that for a generic $m \times m$ matrix ${\bf \Phi}({\bf w})$ with elements $\left\{\Phi_i(w_j)\right\}$ {where the $\Phi_i(x), \, i=1, \ldots, m$ are generic functions,} the following identity holds \cite{Deb:55}  
\begin{equation}
\label{eq:debru}
\underset{{a \leq w_1 < \ldots < w_{m} \leq b} }{\idotsint}  \left|{\bf \Phi}({\bf w})\right|  d{\bf w} = \text{Pf}\left({\bf A}\right)
\end{equation}
where $\text{Pf}\left({\bf A}\right)$ is the Pfaffian, $(\text{Pf}\left({\bf A}\right))^2={\left|{\bf A}\right|}$, and the skew-symmetric matrix $\bf A$ is $m \times m$ for $m$ even, and $(m+1) \times (m+1)$ for $m$ odd, with 
\begin{equation}
\label{eq:aijdebru}
a_{i,j}=\int_a^b \int_a^b \sgn(y-x) \Phi_i(x) \Phi_j(y) dx dy \qquad i,j = 1, \ldots, m .
\end{equation}
For $m$ odd the additional elements are $a_{i,m+1}=-a_{m+1,i}=\int_a^b \Phi_i(x) dx$, $i=1, \ldots, m$, and $a_{m+1,m+1}=0$. 

We then note that, by writing $F_\ell(y)=\int_0^y \Phi_\ell(x) dx$ a primitive of $\Phi_\ell(x)$, we can rewrite \eqref{eq:aijdebru} as 
\begin{align}
\label{eq:aijMC}
a_{i,j}&=\int_a^b \left[ \int_x^b\Phi_i(x) \Phi_j(y) dy - \int_a^x\Phi_i(x) \Phi_j(y) dy \right] dx \nonumber \\
&=\int_a^b \Phi_i(x) \left[ F_j(b)+F_j(a)-2 F_j(x) \right] dx   \\
&=\left[ F_j(b)+F_j(a) \right] \left[ F_i(b)-F_i(a) \right] -2 \int_a^b \Phi_i(x) F_j(x) dx \, . \nonumber
\end{align}
Then, using \eqref{eq:debru} and \eqref{eq:aijMC} in \eqref{eq:th1proof} with $\Phi_i(x)=\gamma_i \, x^{\alpha+i-1} e^{-x}$ and $\gamma_i=1/\Gamma(\alpha+i)$, after some  manipulations we get \eqref{eq:cdfwishartreal} and \eqref{eq:aij}. 

Theorem \ref{th:cdfwishartreal} does not require numerical integration or infinite series. 
In fact, first we observe that  for {an integer $n$} we have  \cite[Ch. 6]{AbrSte:B70} %
\begin{equation}
\label{eq:reggammarecursive}
\reggamma(a+n,x)=\reggamma(a,x)- e^{-x} \sum_{k=0}^{n-1} \frac{x^{a+k}}{ \Gamma(a+k+1)}.
\end{equation}
Therefore, $\reggamma(a,x)$ and $P(a;x,y)$ can be written in closed form when $a$ is integer or half-integer, starting from  
$\reggamma(0,x)=1$ and $\reggamma(1/2,x)= \erf\sqrt{x}$. 
Moreover, using the relation $\reggamma(a+1,x)=\reggamma(a,x)- e^{-x} x^a/\Gamma(a+1)$  in \eqref{eq:aij} gives, after simple manipulations, the iteration  \eqref{eq:aij+1}.
\end{proof}

In summary, the probability that all eigenvalues are within the interval $[a,b]$ is simply obtained, without any numerical integral, by Algorithm~\ref{alg:wis}. 

\begin{algorithm}[!h]
{
\renewcommand{\algorithmicrequire}{\textbf{Input:}}
\renewcommand{\algorithmicensure}{\textbf{Output:}}
\caption{\quad $\psi(a,b)$ for real Wishart matrices}
\label{alg:wis} 
\begin{algorithmic}[0]
\REQUIRE $\nmin, \nmax, a, b$
\ENSURE $\psi(a,b)=\Pr\left\{a \leq \lambda_{\min}({\bf M}), \lambda_{\max}({\bf M}) \leq b \right\}$
 \STATE ${\bf A}={\bf 0}$
 \STATE $\alpha_\ell=(\nmax-\nmin-1)/2+\ell$
 \STATE $g(\alpha_\ell,x)=x^{\alpha_\ell} e^{-x}$
\FOR{$i = 1 \to \nmin-1$}
  	\FOR {$j = i \to \nmin-1$}
		\STATE 
		$$\displaystyle a_{i,j+1}=a_{i,j}+\frac{\Gamma(\alpha_i+\alpha_j) \, 2^{1-\alpha_i-\alpha_j}}{ \Gamma(\alpha_j+1) \Gamma(\alpha_i) } P(\alpha_i+\alpha_j;a,b) - \frac{g(\alpha_j,a/2)+g(\alpha_j,b/2)}{\Gamma(\alpha_j+1)} P\left(\alpha_i;{a}/{2},{b}/{2}\right) $$
  	\ENDFOR
  \ENDFOR
  \IF{$\nmin$ is odd}
     \STATE append to ${\bf A}$ one column according to \eqref{eq:aijoddwishartreal} and a zero row 
  \ENDIF
  \STATE ${\bf A}={\bf A}-{\bf A}^T$
   \RETURN $ K' \, \sqrt{|{\bf A}|}$
\end{algorithmic}
}
\end{algorithm}
For instance, implementing directly the algorithm in Mathematica on a personal computer we obtain the exact value $\psi(a,b)$ for $\nmin=\nmax=500$ in few seconds. 

\subsection{The function $\psi(a,b)$ for real symmetric Gaussian matrices}
The following is a new theorem for GOE matrices.

\def\nmin{n}
\def\Fj{F}

\begin{theorem}
\label{th:cdfGOE}
The probability that all eigenvalues of the real GOE matrix ${\bf M}$ are within the interval $[a,b] \subset \, (-\infty, \infty)$ is
\begin{align}
\label{eq:cdfGOE}
\psi(a,b) 
= K_{GOE}' \,   \sqrt{\left|{\bf A}(a,b)\right|}
\end{align}
with the constant
$$
K_{GOE}' = K_{GOE} \, 2^{\nmin (\nmin+1)/4} \prod_{\ell=1}^{\nmin} \Gamma\left(\ell/2\right) \, .
$$

In \eqref{eq:cdfGOE}, when $\nmin$ is even the elements of the $\nmin \times \nmin$ skew-symmetric matrix ${\bf A}(a,b)$  are
\begin{multline}
\label{eq:aijGOE}
a_{i,j}= 
\left[\Fj_j\left(\frac{b}{\sqrt{2}}\right)+\Fj_j\left(\frac{a}{\sqrt{2}}\right)\right] \Fj_i\left(\frac{a}{\sqrt{2}},\frac{b}{\sqrt{2}}\right) 
- \frac{2}{\Gamma(i/2)} \int_{a/\sqrt{2}}^{b/\sqrt{2}} x^{i-1} e^{-x^2} \Fj_j(x) dx 
\end{multline}
for $i,j=1,\ldots,\nmin$. 

When $\nmin$ is odd, the elements of the $(\nmin+1) \times (\nmin+1)$ skew-symmetric matrix ${\bf A}(a,b)$ are as in \eqref{eq:aijGOE}, with the additional elements
\begin{eqnarray}  \label{eq:aijodd}
  a_{i,\nmin+1}&=&\Fj_i\left(\frac{a}{\sqrt{2}},\frac{b}{\sqrt{2}}\right)   \qquad i=1, \ldots, \nmin  \nonumber \\
a_{\nmin+1,j}&=&-a_{j,\nmin+1} \qquad\qquad j=1, \ldots, \nmin \\  
a_{\nmin+1,\nmin+1}&=&0 \, \nonumber
\end{eqnarray}
%
%
%
where 
\begin{equation}
\label{eq:fjGOE}
\Fj_j(y)=\frac{1}{\Gamma\left(\frac{j}{2}\right)} \int_{0}^{y} x^{j-1} e^{-x^2} dx =\frac{\sgn^j(y)}{2} \reggamma\left(\frac{j}{2},y^2\right)
\end{equation}
and ${\Fj}_j(x,y) \triangleq \Fj_j(y)-\Fj_j(x)$. 

Moreover, the elements $a_{i,j}$ can be computed iteratively, without numerical integration or series expansion, starting from 
\begin{multline*}
a_{2,1}=\frac{1}{4}\left\{\sqrt{2}\left[\erfc(b) -\erfc(a)\right] \right. 
\left. +\left(e^{-a^2/2}+e^{-b^2/2}\right) \left[\erfc\left(\frac{a}{\sqrt{2}}\right)-\erfc\left(\frac{b}{\sqrt{2}}\right)\right] \right\}
\end{multline*}
and using the antisymmetry $a_{j,i}=-a_{i,j}$, together with the iteration
\begin{multline}\label{eq:aij+1GOE}
a_{i,j+2}=a_{i,j}+\frac{\Gamma\left(\frac{i+j}{2}\right) 2^{-(i+j)/2}}{\Gamma(i/2) \Gamma(j/2+1) } \Fj_{i+j}\left(a,b\right) 
- \frac{q(j,a/\sqrt{2})+q(j,b/\sqrt{2})}{2 \Gamma(j/2+1)}\Fj_i\left(\frac{a}{\sqrt{2}},\frac{b}{\sqrt{2}}\right) 
\end{multline}
where $q(j,x)=x^j e^{-x^2}$. 
\end{theorem}

\begin{proof}
We have to integrate the \ac{p.d.f.} in \eqref{eq:jpdfuncorrGOE}. For arbitrary constants $\gamma_i \neq 0$,  we have 
\begin{align*}
\underset{{a \, \leq x_{1} < \cdots < x_{\nmin} \leq \, b}}{\idotsint} & f(x_{\nmin}, \ldots, x_{1}) d{\bf x} = \nonumber \\ 
= K'_{GOE} &\underset{{a/\sqrt{2} \, \leq y_1 < \cdots < y_{\nmin} \leq \, b/\sqrt{2}}}{\idotsint}   \prod_{i=1}^{\nmin} \gamma_i \, e^{-y^2_i}  \prod_{i<j}^{\nmin} (y_j-y_i)  d{\bf y} \nonumber \\
 = K'_{GOE} &\underset{{a/\sqrt{2} \, \leq y_1 < \cdots < y_{\nmin} \leq \, b/\sqrt{2}}}{\idotsint}  \det\left[\left\{ \Phi_i( y_j)\right\}\right]  d{\bf y}  \label{eq:th1GOEproof}
\end{align*}
with $\Phi_i(y)=\gamma_i \, y^{i-1} e^{-y^2}$ and $K'_{GOE}= K_{GOE} 2^{\nmin(\nmin+1)/4} \prod_{i=1}^{\nmin} \gamma_i^{-1}$.
Then, using \eqref{eq:debru} 
  with $\Phi_i(x)=\gamma_i \, x^{i-1} e^{-x^2}$ and $\gamma_i=1/\Gamma(j/2)$, after some manipulations we get \eqref{eq:cdfGOE}. 

Theorem \ref{th:cdfGOE} is amenable to easy evaluation, without numerical integration or infinite series. 
In fact, first we observe that, substituting \eqref{eq:reggammarecursive} in \eqref{eq:fjGOE} we have 
$$
\Fj_{j+2}(y)=\Fj_j(y)-y^j e^{-y^2} \frac{1}{2 \,\Gamma(j/2+1)}.
$$
Using the relation  
$$
\int_0^\beta x^{n-1} e^{-2x^2} dx=2^{-1-n/2} \sgn(\beta)^{n} \,\Gamma\left(\frac{n}{2}\right) P\left(\frac{n}{2},2\beta^2\right)
$$ 
in \eqref{eq:aijGOE} gives, after some manipulations, the iteration  \eqref{eq:aij+1GOE}.
\end{proof}

To build iteratively the upper half of the skew-symmetric matrix ${\bf A}(a,b)$ we just need \eqref{eq:aij+1GOE} and the first two diagonals $a_{i,i}$ and $a_{i,i+1}$. The first diagonal is clearly identically zero due to skew-symmetry, giving $a_{i,i}=0$. The odd first diagonal $a_{i,i+1}$ can be obtained by a zig-zag iteration 
$$a_{1,2} \overset{(a)}{\longrightarrow}  a_{2,1} \overset{(b)}{\longrightarrow} a_{2,3} \overset{(a)}{\longrightarrow} a_{3,2} \overset{(b)}{\longrightarrow} a_{3,4} \overset{(a)}{\longrightarrow} a_{4,3} \cdots$$ 
where steps $(a)$ use skew-symmetry, and steps $(b)$ use \eqref{eq:aij+1GOE}. 
The element $a_{1,2}$ is directly obtained in closed form from \eqref{eq:aijGOE}.

In summary, the probability that all eigenvalues are within the interval $[a,b]$ is simply obtained, without any numerical integral, by Algorithm~\ref{alg:psiGOE}. 
\begin{algorithm}[!h]
{
\renewcommand{\algorithmicrequire}{\textbf{Input:}}
\renewcommand{\algorithmicensure}{\textbf{Output:}}
\caption{\quad $\psi(a,b)$ for real Gaussian matrices (GOE)}
\label{alg:psiGOE} 
\begin{algorithmic}[0]
\REQUIRE $\nmin, \nmax, a, b$
\ENSURE $\psi(a,b)=\Pr\left\{a \leq \lambda_{\min}({\bf M}), \lambda_{\max}({\bf M}) \leq b \right\}$
 \STATE ${\bf A}={\bf 0}$
\STATE $\displaystyle a_{1,2}=-\frac{1}{4}\left\{\sqrt{2}\left[\erfc(b) -\erfc(a)\right] \right. \displaystyle \left. +\left(e^{-a^2/2}+e^{-b^2/2}\right) \left[\erfc\left(\frac{a}{\sqrt{2}}\right)-\erfc\left(\frac{b}{\sqrt{2}}\right)\right] \right\}$ 
\FOR{$i = 1 \to \nmin-2$}
  	\FOR {$j = i \to \nmin-2$}
		\STATE derive $a_{i,j+2}$ from $a_{i,j}$ using \eqref{eq:aij+1GOE}
  	\ENDFOR
		\STATE $a_{i+1,i}=-a_{i,i+1}$
		\STATE derive $a_{i+1,i+2}$ from $a_{i+1,i}$ using \eqref{eq:aij+1GOE} 
		\STATE $a_{i+1,i}=0$
  \ENDFOR
  \IF{$\nmin$ is odd}
     \STATE append to ${\bf A}$ one column according to \eqref{eq:aijodd} and a zero row 
  \ENDIF
  \STATE ${\bf A}={\bf A}-{\bf A}^T$
   \RETURN $ K'_{GOE} \, \sqrt{|{\bf A}|}$
\end{algorithmic}
}
\end{algorithm}

The algorithm can be used to evaluate numerically or symbolically $\psi(a,b)$. Evaluating numerically $\psi(a,b)$ for an arbitrary interval $[a,b]$ requires few seconds for matrices of dimensions $n=500$. The exact expression of $\psi(a,b)$ can be also derived symbolically in closed form. 
Some examples for the probability that all eigenvalues are negative (or all positive, due to symmetry), obtained from Algorithm ~\ref{alg:psiGOE}, are:
\begin{align*}
 n=1 \qquad \psi(-\infty,0)&=\frac{1}{2} \\
 n=2 \qquad \psi(-\infty,0)&=\frac{1}{4} \left(2-\sqrt{2}\right)  \\
 n=3 \qquad \psi(-\infty,0)&=\frac{\pi -2 \sqrt{2}}{4 \pi }  \\
 n=4 \qquad \psi(-\infty,0)&=\frac{\sqrt{\frac{1}{2} \left(9-4 \sqrt{2}\right)} \left(-16-4 \sqrt{2}+7 \pi \right)}{56 \pi }  \\
 n=5 \qquad \psi(-\infty,0)&=\frac{-8-\sqrt{2}+3 \pi }{24 \pi } \\
 n=10 \qquad \psi(-\infty,0)&=\displaystyle \frac{\sqrt{\frac{1}{2} \left(44217-27392 \sqrt{2}\right)}}{183377510400 \pi ^2}  \\ 
  \cdot \big[ \displaystyle 432799744+ & 6251520 \sqrt{2}  \\
  \displaystyle -\big(278413220+ & 1989925 \sqrt{2}\big) \pi +44769900 \pi  ^2\big] \,. \\
\end{align*}
The expressions for $n=1, 2$ and $3$ were already known as reported in \cite{DeaMaj:06,DeaMaj:08}. 
%
%

In \cite{DeaMaj:06,DeaMaj:08} the following asymptotic bound is also derived%
\begin{equation}
\label{eq:inzeroappxdean}
\psi(-\infty,0) \approx e^{-\nmin^2 \ln(3) /4} \, .
\end{equation}
Higher order corrections for large $\nmin$ have been provided in \cite[eq. (19)]{NadMaj:11}, not reported here for space reason. %
Also, the function $\psi(a,b)$ was studied in \cite{DeaMaj:08} for large $n$ and Gaussian matrices, by using a Coulomb gas representation of the distribution of the eigenvalues and arbitrary $a$ and $b$ (see in particular \cite[Eqs. (79), (81) and (82)]{DeaMaj:08}).

By comparing the exact value of $ \psi(-\infty,0)$ with the approximation \eqref{eq:inzeroappxdean} above, we found that the error is exponential in $\nmin$, and well approximated as a factor $10^{-\nmin/6}$. 
We found therefore that an improved approximation for $\psi(-\infty,0)$ is 
\begin{equation}
\label{eq:inzeroappxmc}
\psi(-\infty,0) \approx e^{-\nmin^2 \ln(3) /4-n \ln(10) /6} \, .
\end{equation}
Some values are reported in Table \ref{tab:inzero}.
\begin{table}[tb]
\caption{Probability $\psi(-\infty,0)$ that all eigenvalues are negative, GOE.} 
\label{tab:inzero}
\begin{center}
\begin{tabular}{c l l l l}
\toprule
 $\nmin$ & exact  & approx.   & approx. & approx.  \\
 & (Alg.~\ref{alg:psiGOE}) & \eqref{eq:inzeroappxdean} &  \cite[eq. (19)]{NadMaj:11} & \eqref{eq:inzeroappxmc} \\
 \midrule
  2 & 0.146 & 0.333 & 0.322& 0.155  \\
  5 & 1.40E-4 & 1.04E-3 & 1.91E-3 & 1.53E-4  \\
  10 & 2.27E-14 & 1.18E-12 & 1.23E-12 & 2.54E-14  \\
  50 & 2.43E-307 & 6.30E-299 & 3.31E-304 & 2.92E-307  \\
  100 & 2.72E-1210 & 1.57E-1193 & 1.49E-1206 & 3.39E-1210  \\
  500 & 2.85E-29904 & 8.35E-29821 & 3.88E-29899 & 3.87E-29904  \\
\bottomrule
\end{tabular}
\end{center}
\end{table}

Similar considerations can be done for Wishart matrices, for which approximations for large deviation behavior of the largest eigenvalue are available \cite{VivMaj:07}.  
By combining the large $n$ results for the interval $[0,a]$ in \cite{VivMaj:07} with those for the interval $[b,\infty)$ in \cite{MajViv:12}, it is possible to obtain large $n$ expressions of $\psi(a,b)$ for Wishart matrices. 
In particular,   \cite[eq. (4)]{VivMaj:07} is an approximation for $\psi(0,\nmin)$ for Wishart matrices ${\bf M} \sim {\bf \mathcal{W}}_{\nmin}(\nmin, {\bf I})$. In Table~\ref{tab:inmid} we compare the exact results and the approximation for some values of $\nmin$. 
\begin{table}[tb]
\caption{Probability $\psi(0,\nmin)$, real Wishart, $\nmax=\nmin$.} 
\label{tab:inmid}
\begin{center}
\begin{tabular}{c l l }
\toprule
 $\nmin$ & exact  & approx.  \\
 & (Alg.~\ref{alg:wis})  & \cite[eq. (4)]{VivMaj:07}  \\
 \midrule
  2 & 0.315 & 0.491  \\
  5 & 3.71E-3 &  1.18E-2  \\
  10 & 1.90E-9 &  1.95E-8  \\
  50 & 1.70E-198 &  1.81E-193  \\
  100 & 10.2E-781 &  1.07E-771  \\
  500 & 7.33E-19325 &  6.22E-19275  \\
\bottomrule
\end{tabular}
\end{center}
\end{table}

\def\nmin{n_{\text{min}}}
\def\Fj{\reggamma}


\subsection{The function $\psi(a,b)$ for real multivariate Beta (double Wishart) matrices}
\def\nmin{s}
The following is a new theorem for multivariate real beta matrices in the null case.

\begin{theorem}
\label{th:cdfbeta}
The probability that all eigenvalues of a real multivariate beta matrix ${\bf M}$ are within the interval $[a,b] \subset [0,1]$ is
\begin{equation}
\label{eq:cdfbetareal}
\psi(a,b)=
 K_{MB}' \,   \sqrt{\left|{\bf A}(a,b)\right|}
\end{equation}
with the constant
\begin{equation*}
K_{MB}'=K_{MB} \prod_{\ell=1}^{\nmin} \frac{\Gamma\left(m+\ell\right)}{\Gamma\left(m+\ell+n+1\right)} \,.
\label{eq:K'MB}
\end{equation*}

{In \eqref{eq:cdfbetareal},} when $\nmin$ is even the elements of the $\nmin \times \nmin$ skew-symmetric matrix ${\bf A}(a,b)$  are
\begin{multline}
\label{eq:aijbeta}
a_{i,j}= 
k_i k_j  \left[\betareg(0,a;m+j,n+1)+ \betareg(0,b;m+j,n+1)\right] \\ 
\cdot \betareg(a,b; m+i,n+1) -2 k_i k_j  \int_{a}^{b} x^{m+i-1} (1-x)^{n}  
\betareg(0,x; m+j,n+1) dx 
\end{multline}
for $i,j=1,\ldots,\nmin$, where $k_\ell=\Gamma(m+n+\ell+1)/\Gamma(m+\ell)$.

When $\nmin$ is odd, the elements of the $(\nmin+1) \times (\nmin+1)$ skew-symmetric matrix ${\bf A}(a,b)$ are as in \eqref{eq:aijbeta}, with the additional elements
\begin{eqnarray}  \label{eq:aijoddbeta}
  a_{i,\nmin+1}&=&k_i \betareg(a,b; m+i,n+1)    \qquad i=1, \ldots, \nmin  \nonumber \\
a_{\nmin+1,j}&=&-a_{j,\nmin+1} \qquad\qquad\qquad \, \qquad j=1, \ldots, \nmin \nonumber\\   
a_{\nmin+1,\nmin+1}&=&0 
\end{eqnarray}
%

Moreover, the elements $a_{i,j}$ can be computed iteratively, without numerical integration or infinite series expansion, starting from $a_{i,i}=0$ with the iteration
\begin{multline*}
a_{i,j+1}=a_{i,j}-k_i\left[g_{j+1}(a)+g_{j+1}(b)\right] \betareg(a,b; m+i,n+1) \\ +\frac{2 k_i k_{j+1}}{m+n+j+1} \betareg(a,b; 2m+i+j,2n+2)
\end{multline*}
%
%
for $j=i, \ldots, \nmin-1$, with $g_\ell(x)= x^{m+\ell-1} (1-x)^{n+1} {k_{\ell}}/{(m+n+\ell)}$. 
\end{theorem}

\begin{proof}
Here have to integrate the \ac{p.d.f.} in \eqref{eq:jpdfuncorrnodetMB}. Similarly to the previous, the proof leading to \eqref{eq:aijbeta} uses \eqref{eq:debru}, \eqref{eq:aijdebru} and \eqref{eq:aijMC}, with $\phi_i(x)=k_i x^{m+i-1} (1-x)^n$ and $F_\ell(y)=\int_0^y \Phi_\ell(x) dx=k_\ell \betareg(0,x;m+\ell,n+1)$. 

For the iterative derivation of the elements $a_{i,j}$, we use  the property \cite{Chi:J16}
\begin{equation*}
\label{eq:betarecursive}
\betareg(0,x; a+1, b)=\frac{a}{a+b} \betareg(0,x; a, b)-  \frac{x^a (1-x)^{b}}{a+b}
\end{equation*}
which produces 
%
$F_{j+1}(x)=F_{j}(x)- g_{j+1}(x)$. 
\end{proof}
In summary, the probability that all eigenvalues are within the interval $[a,b]$ is given by Algorithm~\ref{alg:cdfbeta}. 

\begin{algorithm}[!h]
{
\renewcommand{\algorithmicrequire}{\textbf{Input:}}
\renewcommand{\algorithmicensure}{\textbf{Output:}}
\caption{\quad $\psi(a,b)$ for real multivariate beta matrices}
\label{alg:cdfbeta} 
\begin{algorithmic}[0]
\REQUIRE $\nmin, m, n, a, b$
\ENSURE $\psi(a,b)=\Pr\left\{a \leq \lambda_{\min}({\bf M}), \lambda_{\max}({\bf M}) \leq b \right\}$
 \STATE ${\bf A}={\bf 0}$
 \STATE $g_\ell(x)= x^{m+\ell-1} (1-x)^{n+1} {k_{\ell}}/{(m+n+\ell)}$
 \STATE $k_\ell=\Gamma(m+n+\ell+1)/\Gamma(m+\ell)$
 \FOR{$i = 1 \to \nmin-1$}
  	\FOR {$j = i \to \nmin-1$}
		\STATE 
		$\displaystyle a_{i,j+1}=a_{i,j}-k_i\left[g_{j+1}(a)+g_{j+1}(b)\right] \betareg(a,b; m+i,n+1)+\frac{2 k_i k_{j+1}}{m+n+j+1} \betareg(a,b; 2m+i+j,2n+2)$
  	\ENDFOR
  \ENDFOR
  \IF{$\nmin$ is odd}
     \STATE append to ${\bf A}$ one column according to \eqref{eq:aijoddbeta} and a zero row 
  \ENDIF
  \STATE ${\bf A}={\bf A}-{\bf A}^T$
   \RETURN $ K_{MB}' \, \sqrt{|{\bf A}|}$
\end{algorithmic}
}
\end{algorithm}
\noindent
Implementing directly the algorithm in Mathematica on a personal computer, we obtain for example the exact distribution of the largest eigenvalue in less than $0.1$ seconds for all tables in \cite{Pil:67}, \cite[Table B.4]{And:03} and \cite[Table 1]{Joh:09}. 


\def\nmin{n_{\text{min}}}


%
\section{Complex uncorrelated and correlated Wishart and Hermitian Gaussian matrices}
\label{sec:exactcomplex}
%
The analysis for complex random matrices is easier than for the real case, and in fact some important results are known since many years for complex multivariate Beta matrices and for  uncorrelated complex Wishart \cite{Kha:64}. 
A general methodology which can be applied also to correlated complex Wishart (i.e., with covariance matrix not proportional to the identity matrix) is given in \cite{ChiZan:C08} and here specialized to provide $\psi(a,b)$ in several situations.

\def\nmin{n}

\subsection{Complex Wishart matrices}
Assume a Gaussian complex $\nmin \times \nmax$ matrix ${\bf X}$ with \ac{i.i.d.} columns, each circularly symmetric with covariance ${\bf \Sigma=I}$, and $\nmax \geq \nmin$.  
%
%
Denoting 
 $\td{\Gamma}_{\Nmin}(m)=\pi^{\Nmin(\Nmin-1)/2} \prod_{i=1}^{\Nmin}(m-i)!$, the joint \ac{p.d.f.} of the (real) ordered eigenvalues $\lambda_1 \geq \lambda_2 \ldots \geq\lambda_{\Nmin} \geq 0$ of the complex Wishart matrix ${\bf M=X X}^H \sim {\bf {\mathcal{CW}}}_{\nmin}(\nmax, {\bf I})$ (identity covariance) is well known to be \cite{Ede:88,Joh:01,ChiWinZan:J03} 
\begin{equation}\label{eq:jpdfuncorrnodetcomplex}
f(x_{1}, \ldots, x_{\Nmin}) = K \,
     \prod_{i=1}^{\Nmin}e^{-x_{i}}x_{i}^{\nmax-\Nmin} \prod_{i<j}^{\Nmin} \left(x_{i}-x_{j}\right)^2 
\end{equation}
where {$x_1 \geq x_2 \geq \cdots \geq x_{\Nmin} \geq 0$ and $K$ is a normalizing constant given by
\begin{equation}
1/K =\prod_{i=1}^{\nmin} (\nmax-i)! (\nmin-i)! \,.
\label{eq:Kwishartcomplex}
\end{equation}

 Assume now a Gaussian complex $\nmin \times \nmax$ matrix ${\bf X}$ with \ac{i.i.d.} columns, each circularly symmetric with covariance ${\bf \Sigma}$, and $\nmax \geq \nmin$.  
The joint distribution of the ordered eigenvalues 
 of ${\bf M =X X}^H \sim {\bf \mathcal{CW}}_{\nmin}(\nmax, {\bf \Sigma})$ has firstly been found in \cite{ChiWinZan:J03} as follows. 
\begin{lemma} 
Let ${\bf M} \sim {\mathcal{CW}}_n(N, \mat{\Sigma})$ be a complex Wishart matrix, $N\geq n$. Denote $\sigma_1 > \sigma_2 >
\ldots > \sigma_{\nmin} > 0$ the ordered
eigenvalues of $\bf \Sigma$. Then, the joint p.d.f. of the ordered eigenvalues of  ${\bf M}$ is 
\begin{align}\label{eq:jpdfcorrsimple}
f(x_{1}, \ldots, x_{\nmin})
    &= K_{{\bf \Sigma}} 
    \left|\mb{E}\left({\bf x},\bm{\sigma}\right)\right|
        \cdot \prod_{i<j}^{\nmin} (x_i-x_j)
        \cdot \prod_{j=1}^{\nmin} x^{\nmax-\nmin}_j 
\end{align}
where $\mb{E}\left({\bf x},\bm{\sigma}\right)=\left\{e^{-x_i/\sigma_j}\right\}_{i,j=1}^{\nmin}$ and
\begin{equation}
\label{eq:ksigma}
1/K_{{\bf \Sigma}} = \prod_{i<j}^{\nmin} (\sigma_i-\sigma_j) \prod_{i=1}^{\nmin} \sigma_i^{\nmax-\nmin+1} (\nmax-i)!  \, .
\end{equation}
\end{lemma}
\begin{proof}
See \cite{ChiWinZan:J03}.
\end{proof}
The analysis in \cite{ChiWinZan:J03} has been extended to the case where ${\bf \Sigma}$ has eigenvalues of arbitrary multiplicity and to the marginal eigenvalues distribution in  \cite{ChiZan:C08,ChiWinShi:J10,ZanChiWin:J09}. 

In particular, when ${\bf \Sigma}$ is spiked with $\sigma_1>\sigma_2=\sigma_3=\sigma_4=\cdots=\sigma_n$, we have the following result.
\begin{lemma}\label{lemmaCWspiked}
Let ${\bf M} \sim {\mathcal{CW}}_n(N, \mat{\Sigma})$ be a complex Wishart matrix, $N\geq n$. Denote $\sigma_1 > \sigma_2 =
\ldots = \sigma_{\nmin} > 0$ the ordered
eigenvalues of $\bf \Sigma$ (spiked covariance matrix). Then, the joint p.d.f. of the ordered eigenvalues of  ${\bf M}$ is 
\begin{align}\label{eq:jpdfcorrspiked}
f(x_{1}, \ldots, x_{\nmin})
    &= K_{1} 
    \left|\mb{E}\left({\bf x},\bm{\sigma}\right)\right|
        \cdot \prod_{i<j}^{\nmin} (x_i-x_j)
        \cdot \prod_{j=1}^{\nmin} x^{\nmax-\nmin}_j 
\end{align}
where $\mb{E}\left({\bf x},\bm{\sigma}\right)$ has elements 
\begin{align*}
\label{eq:eijcomplexcorr}
e_{i,j}=\left\{
  \begin{array}{ll}
    \displaystyle e^{-x_i/\sigma_1}  & j=1\\
   \displaystyle x_i^{n-j} e^{-x_i/\sigma_2}  & j=2, \ldots,\nmin
  \end{array}
\right.
\end{align*}
and
\begin{equation*}
\label{eq:ksigmaspiked}
\frac{1}{K_{1}} =  \sigma_1^{\nmax-\nmin+1} \sigma_2^{(\nmax-1)(\nmin-1)} (\sigma_1-\sigma_2)^{\nmin-1} \, \prod_{i=1}^\nmin (\nmax-i)! \prod_{\ell=2}^{\nmin-2} \ell ! \,\,\,\, .
\end{equation*}
\end{lemma}
\begin{proof}
This is a particular case of \cite[Lemma 6]{ChiWinShi:J10}.
\end{proof}

Below we report $\psi(a,b)$ for complex Wishart matrices. 
\begin{theorem}
\label{th:cdfwishartcomplex}
For complex Wishart matrices ${\bf M} \sim {\bf \mathcal{CW}}_{n}(N, {\bf \Sigma})$, $\nmax \geq \nmin$, the probability that all eigenvalues are within $[a,b] \subset [0, \infty)$ is given below, depending on the covariance $ {\bf \Sigma}$. 
\begin{enumerate}
\item
For the uncorrelated complex Wishart matrix ${\bf M} \sim {\bf \mathcal{CW}}_{n}(N, {\bf I})$: 
\begin{equation*}
\label{eq:cdfwishartcomplex}
\psi(a,b) =K \left|{\bf A}(a,b)\right|
\end{equation*}
where the elements of the $\nmin \times \nmin$ matrix ${\bf A}(a,b)$ are
\begin{align*}
\label{eq:aijcomplex}
a_{i,j} &=  \int_{a}^{b} t^{\nmax+\nmin-i-j} e^{-t} dt 
=\gamma\left(\nmax+\nmin-i-j+1; a, b \right)  \,
\end{align*}
and $K$ is given in \eqref{eq:Kwishartcomplex}. %
%
%

\bigskip
\item For the correlated complex Wishart matrix ${\bf M} \sim {\bf \mathcal{CW}}_{n}(N, {\bf \Sigma})$ where ${\bf \Sigma}$ has distinct eigenvalues $\sigma_1 > \sigma_2 > \cdots > \sigma_{\nmin} $: 
\begin{equation*}
\label{eq:cdfwishartcomplexcorr}
\psi(a,b)=
K_{\mat{\Sigma}} \left|{\bf A}(a,b)\right|
\end{equation*}
where the elements of the $\nmin \times \nmin$ matrix ${\bf A}(a,b)$ are
\begin{align*}
\label{eq:aijcomplexcorr}
a_{i,j} &=  \int_{a}^{b} t^{\nmax-i} e^{-t/\sigma_j} dt 
=\sigma_j^{\nmax-i+1} \gamma\left(\nmax-i+1;\frac{a}{\sigma_j},\frac{b}{\sigma_j}\right)  \, 
\end{align*}
 and $K_{\mat{\Sigma}}$ is given in \eqref{eq:ksigma}.

\bigskip
\item For the correlated complex Wishart matrix ${\bf M} \sim {\bf \mathcal{CW}}_{n}(N, {\bf \Sigma})$ with a spiked covariance ${\bf \Sigma}$ having eigenvalues $\sigma_1>\sigma_2=\sigma_3=\sigma_4=\cdots=\sigma_n$: 
\begin{equation*}
\label{eq:cdfwishartcomplexcorrspik}
\psi(a,b)=
K_{1} \left|{\bf A}(a,b)\right|
\end{equation*}
where the elements of the $\nmin \times \nmin$ matrix ${\bf A}(a,b)$ are
%
%
%
\begin{align*}
a_{i,1}&= \int_{a}^{b} t^{\nmax-i} e^{-t/\sigma_1} dt 
=\sigma_1^{\nmax-i+1} \gamma\left(\nmax-i+1;\frac{a}{\sigma_1},\frac{b}{\sigma_1}\right)
\end{align*}
and, for $j=2, \ldots,\nmin$,
\begin{align*}
a_{i,j}&=  \int_{a}^{b} t^{\nmax+\nmin-i-j} e^{-t/\sigma_2} dt 
= \sigma_2^{\nmax+\nmin-i-j+1} \gamma\left(\nmax+\nmin-i-j+1;\frac{a}{\sigma_2},\frac{b}{\sigma_2}\right) \,.
\end{align*}
The constant $K_{1}$ is given in Lemma~\ref{lemmaCWspiked}. 
\end{enumerate}
\end{theorem}
\begin{proof}
This theorem can be obtained by specializing \cite[Theorem 7]{ChiZan:C08}. More precisely, we first rewrite the p.d.f.'s in \eqref{eq:jpdfuncorrnodetcomplex}, \eqref{eq:jpdfcorrsimple}, and \eqref{eq:jpdfcorrspiked} as product of determinants of two matrices, by expressing  $\prod_{i<j}(x_i-x_j)$ as the determinant of a Vandermonde matrix. Then, applying \cite[Theorem 7]{ChiZan:C08}, after some simplifications we get the theorem.
\end{proof}
Note that the uncorrelated case 1) in the previous theorem can be seen as an extension of \cite[eq. (6)]{Kha:64}.

We remark that approximations and asymptotics for spiked Wishart have also been studied in recent literature (see e.g. \cite{Joh:01,BaiSil:06,Nad:08}).
%

\subsection{Hermitian Gaussian matrices (GUE)}
The Gaussian Unitary Ensemble (GUE) is composed of complex Hermitian random matrices with \ac{i.i.d.} ${\mathcal{CN}}(0, 1/2)$  entries on the upper-triangle, and ${\mathcal{N}}(0, 1/2)$ on the main diagonal \cite{TraWid:09}.  

The following theorem applies to the GUE.
\begin{theorem}
\label{th:cdfGUE}
The probability that all eigenvalues of a $n \times n$ GUE matrix ${\bf M}$ are within the interval $[a,b] \subset (-\infty,\infty)$ is
\begin{equation}
\label{eq:cdfGUE}
\psi(a,b)=
K_{GUE} \left|{\bf A}(a,b)\right|
\end{equation}
where the elements of the $n \times n$ matrix ${\bf A}(a,b)$ are
\begin{align*}
\label{eq:aijGUE}
a_{i,j} =&  \int_{a}^{b} t^{i+j-2} e^{-t^2} dt  \\
= & \frac{1}{2} \Gamma\left(\frac{i+j-1}{2} \right) \Big[\reggamma\left(\frac{i+j-1}{2}, b^2 \right)\sgn(b)^{i+j-1}  
- \reggamma\left(\frac{i+j-1}{2}, a^2 \right)\sgn(a)^{i+j-1}\Big]  \nonumber
\end{align*}
and $K_{GUE}=2^{n(n-1)/2} (\pi^{n/2} \prod_{i=1}^{n} \Gamma[i])^{-1}$ is a normalizing constant.
\end{theorem} 
\begin{proof}
As for the complex white Wishart, this theorem for GUE is easily derived from known results. In fact, the joint distribution of the ordered eigenvalues can be written as \cite{TraWid:09}
\begin{equation} \label{eq:jpdfuncorrGUE}
f(x_1,\ldots, x_n) = K_{GUE}  \prod_{i<j}^{\Nmin} \left(x_{i}-x_{j}\right)^2 
     \prod_{i=1}^{n}e^{-x_{i}^2} \, .
\end{equation}
Then, by using \cite[Corollary 2]{ChiWinZan:J03} with $\Psi_i(x_j)=\Phi_i(x_j)=x_j^{i-1}, \xi(x)=e^{-x^2}$ we get the result. 
\end{proof}
%

\subsection{Complex multivariate beta (double Wishart) matrices}
When ${\bf X, Y} $ are two independent {\emph{complex}} Gaussian, the analogous of \eqref{eq:jpdfuncorrnodetMB} is the complex multivariate beta, where the joint distribution of the eigenvalues is \cite{Kha:64}
\begin{equation}\label{eq:jpdfuncorrnodetcomplexMB}
f(\x_{1}, \ldots, \x_{s}) = K_{MB} \, 
     \prod_{i=1}^{s} \x_{i}^{m} (1-\x_{i})^n
    \cdot    \prod_{i<j}^{s}
    \left(\x_{i}-\x_{j}\right)^2 
\end{equation}
with $1 > \x_{1} \geq \x_{2} \cdots \geq \x_{s} > 0$, and 
$$K_{MB}=\prod _{i=1}^s  \frac{ \Gamma \left({m+ n+s+i}\right)}{ \Gamma
   \left({i}\right) \Gamma \left({i+m} \right) \Gamma \left({i+n} \right)} \,.$$ 
%
Therefore, by applying \cite[Corollary 2]{ChiWinZan:J03} we have for a complex multivariate Beta matrix ${\bf M}$ 
\begin{equation}
\label{eq:psiMBcomplex}
\psi(a,b)=
 K_{MB} \,   \left|{\bf A}(a,b)\right| \,
\end{equation}
where the elements of the $s \times s$ matrix ${\bf A}(a,b)$  are
\begin{equation*}
\label{eq:aijcomplexMB}
a_{i,j}=
\betareg(a, b; m+i+j-1,n+1) 
\end{equation*}
for $i,j=1,\ldots,s$.

Note that \eqref{eq:psiMBcomplex} can be seen as an extension of \cite[eq. (3)]{Kha:64}.

    
\def\nmin{n}

\section{Asymptotics and approximations}
\label{sec:approx} 
In this section we study $\psi(a,b)$ for large white Wishart and Gaussian matrices. 
To this aim we make the following three observations.
\begin{enumerate}
\item 
The statistical dependence between the largest and the smallest eigenvalues passes through the intermediate $\nmin-2$ eigenvalues.
Consequently, in the limit for $\nmin \to \infty$ the largest eigenvalue and the smallest eigenvalue are independent. Thus, for large matrix sizes we have 
\begin{equation}
\label{eq:psiindep}
\Pr\left\{{a} \leq \lambda_{\min}({\bf M}), \lambda_{\max}({\bf M}) \leq {b}\right\} 
\approx \Pr\left\{{a} \leq \lambda_{\min}({\bf M})\right\}  \Pr\left\{\lambda_{\max}({\bf M}) \leq {b}\right\} 
\end{equation}
which is like to say $\psi(a,b)\approx\psi(a,\infty) \psi(-\infty,b)$.

\item The distribution of the smallest and largest eigenvalues of white Wishart and Gaussian matrices for small deviations from the mean tend to a properly scaled and shifted Tracy-Widom distribution \cite{TraWid:94,TraWid:96,Joha:00,Joh:01,Elk:03, TraWid:09, FelSod:10, Ma:12}.  
%
More precisely, focusing for example on white Wishart matrices ${\bf M} \sim {\bf \mathcal{W}}_{p}(m, {\bf I})$ or ${\bf M} \sim {\bf \mathcal{CW}}_{p}(m, {\bf I})$,
 when 
 $m, p \rightarrow \infty$  and $m/p \rightarrow \gamma  \in [0,\infty]$ 
\begin{equation}
\label{eq:l1pca}
\frac{\lambda_{\max}(\mat{M}) -  \mu_{mp} }{\sigma_{mp}} \overset{{\mathcal{D}}}{\longrightarrow} {\W_{\beta}}
\end{equation}
%
where $\W_{\beta}$ denotes the Tracy-Widom random variable of order $\beta$ whose \ac{CDF} will be indicated with $F_\beta(x)$, 
  and\footnote{For small sizes a better approximation is obtained by using slightly different values, like $m-1/2$ and $p-1/2$ for the real case, instead of $m, p$. However, since $\psi(a,b)$ can be computed exactly as shown in Section~\ref{sec:exactreal} and \ref{sec:exactcomplex}, asymptotic expressions are interesting only for large dimensions for which these corrections are irrelevant.} 
\begin{align}
\label{eq:musigma}
\mu_{mp}=\left( \sqrt{m}+\sqrt{p}\right)^2 \quad \sigma_{mp}=\sqrt{\mu_{mp}} \left( \frac{1}{\sqrt{p}}+\frac{1}{\sqrt{m}}\right)^{\frac{1}{3}} .
\end{align}
In the previous expressions $\beta=1$ and $\beta=2$ for real and complex matrices, respectively. 
  
Similarly, for the smallest eigenvalue we can use the results in \cite{FelSod:10} or \cite{Ma:12}. For example, in \cite{FelSod:10} it is shown that when $m, p \rightarrow \infty$  and $m/p \rightarrow \gamma  \in (1,\infty)$,
\begin{equation}
\label{eq:lppcafed}
-\frac{\lambda_{\min}(\mat{M}) - \mu^-_{mp}}{\sigma^-_{mp}} \overset{{\mathcal{D}}}{\longrightarrow} {\W_{\beta}}
\end{equation}
with
\begin{align}
\label{eq:musigmameno}
\mu^-_{mp}=\left( \sqrt{m}-\sqrt{p}\right)^2 \quad \sigma^-_{mp}=\sqrt{\mu^-_{mp}} \left( \frac{1}{\sqrt{p}}-\frac{1}{\sqrt{m}}\right)^{\frac{1}{3}} .
\end{align}
%
%
%
Note that the variance of the smallest eigenvalue is smaller than that of the largest eigenvalue. 

For the Gaussian ensemble we have\footnote{For GOE and GUE the smallest and largest eigenvalues have symmetrical distributions.}
 for $n \to \infty$ \cite{TraWid:94,TraWid:96,TraWid:09}
\begin{align*}
\label{eq:l1GOEGUE}
\frac{\lambda_{\max}(\mat{M})  -  \mu_{n} }{\sigma_{n}} &\overset{{\mathcal{D}}}{\longrightarrow} {\W_{\beta}} \\
-\frac{\lambda_{\min}(\mat{M})  - \mu^-_{n} }{\sigma^-_{n}} &\overset{{\mathcal{D}}}{\longrightarrow} {\W_{\beta}} 
\end{align*}
with 
\begin{equation}
\label{eq:musigmaG}
\mu_{n}=2\sigma_0 \sqrt{n} \quad \mu^-_{n}=-\mu_{n} \quad \sigma_{n}=\sigma^-_{n}=\sigma_0 (n)^{-1/6} \, .
\end{equation}

\item The Tracy-Widom distribution can be accurately approximated by a scaled and shifted gamma distribution
\begin{equation}
\label{eq:appxmc} 
\W_{\beta}  \simeq \Gamma(\kappx,\tappx)-\aappx
\end{equation}
where $\aappx$ is a constant, and $\Gamma(\kappx,\tappx)$ denotes a gamma \ac{r.v.} with shape parameter $\kappx$ and scale parameter $\tappx$  \cite{Chi:J14}. 
Thus the \ac{CDF} 
 of $\W_{\beta}$ is accurately approximated by an incomplete gamma function as:
\begin{equation}
\label{eq:Fbetaappx}
\Pr\left\{\W_\beta \leq x\right\}= F_{\beta}(x) \simeq 
  \reggamma\left(\kappx,\frac{\left(x+\aappx\right)^+}{\tappx}\right) 
\end{equation}
where $(x)^+=\max\{0,x\}$ denotes the positive part. 
The parameters $\kappx,\tappx, \aappx$ 
 are reported in Table~\ref{tab:par} \cite{Chi:J14}.
\begin{table}[tb]
\caption{Parameters for approximating $\W_{\beta}$ with $\Gamma[\kappx,\tappx]-\aappx$.} 
\label{tab:par}
\begin{center}
\begin{tabular}{c l l l} \toprule
 &  $\W_1$ &  $\W_2$ &  $\W_4$ \\
 \midrule 
$ \kappx $&  46.446 & 79.6595 & 146.021 \\ 
 $\tappx$ & 0.186054 & 0.101037& 0.0595445\\
$ \aappx$ & 9.84801 & 9.81961 & 11.0016 \\
\bottomrule
\end{tabular}
\end{center}
\end{table}
\end{enumerate}


Thus, putting the gamma approximation in \eqref{eq:l1pca}, \eqref{eq:lppcafed} we have for white Wishart, GOE and GUE matrices 
%
\begin{equation}
\Pr\left\{\lambda_{\max}({\bf M})<{b}\right\} \to F_\beta\left(\frac{b-\mu}{\sigma}\right) 
\simeq \reggamma\left(\kappx, \frac{\left(\aappx+(b-\mu)/\sigma \right)^{+}}{\tappx } \right)  \label{eq:lmaxTWMC}
\end{equation}
\begin{equation}
\Pr\left\{\lambda_{\min}({\bf M})>{a}\right\} \to F_\beta\left(-\frac{a-\mu^-}{\sigma^-}\right) 
\simeq \reggamma\left(\kappx, \frac{\left(\aappx-(a-\mu^-)/\sigma^- \right)^{+}}{\tappx } \right) \label{eq:lminTWMC} 
\end{equation}
where $\mu, \sigma, \mu^-, \sigma^-$ are given by \eqref{eq:musigma}, \eqref{eq:musigmameno}, \eqref{eq:musigmaG}, and the parameters $\kappx,\tappx, \aappx$ are given in Table~\ref{tab:par}. 
These can be used in \eqref{eq:psiindep} to give 
%
%
%
\begin{equation}
\label{eq:psiappx}
\psi(a,b) \simeq 
\reggamma\left(\kappx, \left(\frac{\aappx}{\tappx}+\frac{b-\mu}{\tappx\sigma} \right)^{+} \right) \reggamma\left(\kappx, \left(\frac{\aappx}{\tappx}-\frac{a-\mu^-}{\tappx\sigma^-} \right)^{+} \right) \,.
\end{equation}
%
 %
Finally, we observe that \eqref{eq:psiappx} can be used not only for white Wishart and Gaussian symmetric/Hermitian matrices, but also for a wider class of matrices, due to the universality of the Tracy-Widom laws for the smallest and largest eigenvalues of large random matrices \cite{Sos:02,Pec:09,FelSod:10}. 

\section{Probability the all eigenvalues are within the support of the limiting Mar{\v{c}}enko-Pastur and Wigner spectral distribution}
%
\def\nmax{m}
\def\nmin{p}
Under quite general conditions, for a matrix $\mat{M}=\mat{X}\mat{X}^H$ where $\mat{X}$ is $(\nmin \times \nmax)$ with \ac{i.i.d.} entries with zero mean and variance $\sigma^2=1$, the Mar{\v{c}}enko-Pastur law gives the asymptotic \ac{p.d.f.} of an unordered\footnote{This can be seen as the distribution of a randomly picked eigenvalue, or of the arithmetic mean of all eigenvalues. In physics literature, this is related to the fraction of eigenvalues below a given value (spectral distribution).} eigenvalue $\lambda=\lambda({\bf M})$ for large $\nmin, \nmax$ with a fixed ratio $\nmin/\nmax \leq 1$ as\footnote{If $\nmin/\nmax>1$ the matrix has $\nmin-\nmax$ zero eigenvalues, so the distribution has an additional point of mass $1-\nmax/\nmin$ in $0$.} 
\begin{align*}
f(\lambda)= \left\{
  \begin{array}{ll}
  \displaystyle\frac{1}{2\pi \nmin \lambda} {\sqrt{(\tilde{b} - \lambda)(\lambda - \tilde{a})}}{} &\tilde{a} \leq \lambda \leq \tilde{b} \\
 0 & otherwise
  \end{array}\right. 
\end{align*}
%
where $\tilde{a} = \left(\sqrt{\nmax} - \sqrt{\nmin}\right)^2$ and $\tilde{b} = \left(\sqrt{\nmax} + \sqrt{\nmin}\right)^2$, and $[\tilde{a}, \tilde{b}]$ is the support of the Mar{\v{c}}enko-Pastur law \cite{MarPas:67,BaiSil:06}. 

Also, for increasing $\nmin, \nmax$ it has been proved that the the minimum and maximum eigenvalues converge to the edges of the Mar{\v{c}}enko-Pastur law  $\lambda_{\min}({\bf M}) \to \tilde{a}$ and $\lambda_{\max}({\bf M}) \to \tilde{b}$ \cite{BaiSil:06}. 
Note that when the entries of $\mat{X}$ are Gaussian the matrix $\mat{M}$ is white Wishart.

Similarly, under quite general conditions, for a Wigner matrix $\mat{M}=\mat{X}+\mat{X}^H$ where $\mat{X}$ is $(\nmin \times \nmin)$ with \ac{i.i.d.} entries with zero mean and variance $\sigma^2=1/4$, the Wigner semicircle law gives the asymptotic \ac{p.d.f.} of an unordered eigenvalue $\lambda=\lambda({\bf M})$ for large $\nmin$ as 
\begin{align*}
\label{eq:semicicle}
f(\lambda)= \left\{
  \begin{array}{ll}
  \displaystyle\frac{1}{\pi \nmin} \sqrt{2\nmin - \lambda^2} & |\lambda| \leq \sqrt{2 \nmin}\\
 0 & otherwise
  \end{array}\right.
\end{align*}
where $[-\sqrt{2 \nmin}, \sqrt{2 \nmin}]$ is the support of the semicircle law. Again, for increasing $\nmin$ it has been proved that the minimum and maximum eigenvalues converge to the edges of the semicircle support  $\lambda_{\min}({\bf M}) \to -\sqrt{2 \nmin}$ and $\lambda_{\max}({\bf M}) \to \sqrt{2 \nmin}$ \cite{BaiSil:06}. 

So, we could be tempted to think that for increasing matrix sizes all eigenvalues are within the Mar{\v{c}}enko-Pastur or semicircle supports with probability tending to one. However, this is not the case, 
 as proved in the following theorem. 
\begin{theorem}
\label{th:edges}
For increasingly large matrices, the probability that all eigenvalues of Wishart and Gaussian Hermitian matrices are within the Mar{\v{c}}enko-Pastur and semicircle supports tends to $F^2_1(0) = 0.6921$ and  $F^2_2(0) = 0.9397$ for the real and complex cases, respectively. 

More precisely, 
 we have the following results. 
\begin{enumerate}
\item
Let ${\bf M} \sim {\bf \mathcal{W}}_{\nmin}(\nmax, {\bf I})$ be a real Wishart matrix with $\nmax>\nmin$.  When $\nmax, \nmin \rightarrow \infty$  and $\nmax/\nmin \rightarrow \gamma  \in (1,\infty)$, the probability that all eigenvalues are within the Mar{\v{c}}enko-Pastur support is   
\begin{equation*}
\psi\left((\sqrt{\nmax}-\sqrt{\nmin})^2, (\sqrt{\nmax}+\sqrt{\nmin})^2\right) \to F^2_1(0) = 0.6921 \, .
\end{equation*}


\item 
Let ${\bf M} \sim {\bf \mathcal{CW}}_{\nmin}(\nmax, {\bf I})$ be a complex Wishart matrix with $\nmax>\nmin$.  
When $\nmax, \nmin \rightarrow \infty$  and $\nmax/\nmin \rightarrow \gamma  \in (1,\infty)$, the probability that all eigenvalues are within the Mar{\v{c}}enko-Pastur support is
\begin{equation*}
\label{eq:psiMPcomplex}
\psi\left((\sqrt{\nmax}-\sqrt{\nmin})^2, (\sqrt{\nmax}+\sqrt{\nmin})^2\right)  \to F^2_2(0) = 0.9397 \, .
\end{equation*}

\item
Let ${\bf M} $ be a $(\nmin \times \nmin)$ real symmetric GOE matrix.  When $\nmin \rightarrow \infty$ the probability that all eigenvalues are within the semicircle support is   
\begin{equation*}
\label{eq:psicircGOE}
\psi\left(-\sqrt{2 \nmin}, \sqrt{2 \nmin}\right) \to F^2_1(0) = 0.6921 \, .
\end{equation*}


\item 
Let ${\bf M} $ be a $(\nmin \times \nmin)$ complex symmetric GUE matrix.  When $\nmin \rightarrow \infty$ the probability that all eigenvalues are within the semicircle support is   
\begin{equation*}
\label{eq:psicircGUE}
\psi\left(-\sqrt{2 \nmin}, \sqrt{2 \nmin}\right) \to F^2_2(0) = 0.9397  \, .
\end{equation*}

\end{enumerate}
\end{theorem}
\begin{proof}
Let us consider the Wishart case. 
Then, using \eqref{eq:psiindep} and observing that the Mar{\v{c}}enko-Pastur edges are the constants $\mu_{mp}$ and  $\mu^-_{mp}$ appearing in \eqref{eq:l1pca} and \eqref{eq:lppcafed}, 
 we get the results. For GOE/GUE the limiting distribution of the extreme eigenvalues is still a shifted (of an amount equal to the circular law edges) and scaled $\W_\beta$, so the same reasoning leads to the results. 
\end{proof}
%
This theorem is valid not only for matrices derived from Gaussian measurements, but for the much wider class of matrices for which \eqref{eq:l1pca} and \eqref{eq:lppcafed} apply \cite{Sos:02,Pec:09,FelSod:10}. 

Note that the gamma approximation \eqref{eq:Fbetaappx} for the Tracy-Widom law gives $F^2_1(0) \simeq \reggamma^2\left(\kappx, {\aappx}/{\tappx }\right) = 0.8312^2 = 0.691$ and $F^2_2(0) \simeq  \reggamma^2\left(\kappx, {\aappx}/{\tappx }\right) = 0.96945^2 = 0.9398$.

Exact values for $\psi(\tilde{a}, \tilde{b})$ for finite dimension real Wishart matrices, as obtained by Algorithm~\ref{alg:wis}, are reported in Table~\ref{tab:psiabtw1}, together with the asymptotic value.

Finally, since for Wishart matrices the smallest and largest eigenvalues have different variances, asymptotically we can leave the same probability at the left and right side by moving $t \, \sigma^-_{ms}$ on the left and $t \, \sigma_{ms}$ on the right of the limiting Mar{\v{c}}enko-Pastur support. In fact, we have
\begin{align*}
\Pr\left\{\lambda_{\max}(\mat{M}) > \mu_{ms}+t \, \sigma_{ms} \right\} & \to 1-F_\beta(t) \\ 
\Pr\left\{\lambda_{\min}(\mat{M}) < \mu^-_{ms}-t \, \sigma^-_{ms} \right\} & \to 1-F_\beta(t)  
\end{align*}
and thus
\begin{align*}
\psi(\mu^-_{ms}-t \, \sigma^-_{ms}, \mu_{ms}+t \, \sigma_{ms}) & \to (1-F_\beta(t))^2
\end{align*}
with, for instance, $F_1(0) \simeq 83\%, \, F_1(1)\simeq 95\%, \, F_1(2) \simeq 99\%, \, F_1(3) \simeq 99.8 \%$. 

\begin{table}[tb]
\caption{Probability $\psi(\tilde{a}, \tilde{b})$ that all eigenvalues of a real Wishart matrix are within the Mar{\v{c}}enko-Pastur edges for $\nmin/\nmax=2/3, 1/2, 1/5, 1/10$. Numerical values for finite $\nmin$ obtained by Algorithm~\ref{alg:wis}, and for $\nmin=\infty$ by Theorem~\ref{th:edges}.} 
\label{tab:psiabtw1}
\begin{center}
\begin{tabular}{c c c c c}
\toprule
\multicolumn{5}{c}{$\psi(\tilde{a}, \tilde{b})$} \\ 
\cline{2-5} 
 &  $\nmin/\nmax=2/3$ & $\nmin/\nmax=1/2$ & $\nmin/\nmax=1/5$ & $\nmin/\nmax=1/10$ \\
$\nmin$ & & & & \\
 \midrule
 10 &  0.7678 & 0.7645 & 0.7625 & 0.7624\\
 20 &  0.7499 & 0.7483 & 0.7476 & 0.7477\\
 50 &  0.7332 & 0.7326 & 0.7327 & 0.7329 \\
 100 &  0.7239 & 0.7238 & 0.7242 & 0.7244\\
 200 &  0.7169 & 0.7171 & 0.7175 & 0.7177\\
 500 &  0.7101 & 0.7103 & 0.7108 & 0.7109\\
 $\infty$ &  0.6921 & 0.6921 & 0.6921 & 0.6921 \\
 \bottomrule
\end{tabular}
\end{center}
\end{table}
%


\section{Application to compressed sensing}
Assume that we want to solve the system
\begin{equation}
\label{eq:cs}
\vect{y}=\mat{A}\vect{x}
\end{equation}
where $\vect{y} \in \mathbb{R}^m$ and $\mat{A} \in \mathbb{R}^{m \times n}$ are known, the number of equations is  $m <n$, and  $\vect{x} \in \mathbb{R}^n$ is the unknown. Since $m<n$, we can think of $\vect{y}$ as a compressed version of $\vect{x}$. Without other constraints the system is underdetermined, and there are infinitely many distinct solutions $\vect{x}$ satisfying \eqref{eq:cs}. If we assume that at most $s$ elements of $\vect{x}$ are non-zero (i.e., the vector is $s$-sparse), and $s < m$, then there is only one solution (the right one) to \eqref{eq:cs}, provided that all possible submatrices consisting of $2s$ columns of  $\mat{A}$ are maximum rank ($2s$). However, even when this condition is satisfied, finding the solution of \eqref{eq:cs} subject to $||\vect{x}||_0 \le s$, where the $\ell_0$-``norm'' $||\cdot ||_0$ is the number of non-zero elements, is computationally prohibitive. A computationally much easier problem is to find a $\ell_1$-norm minimization solution $\vect{x}=\left\{\arg\min_{\tilde{\vect{x}}} ||\tilde{\vect{x}}||_1 \, : \, \vect{y}=\mat{A}\tilde{\vect{x}} \right\}$. In \cite{CanTao:05} it is proved that, under some more strict conditions on $\mat{A}$, the solution provided by $\ell_1$-norm minimization is the same as that of the $\ell_0$-``norm'' minimization. More precisely, for integer $s$  define the isometry constant of a matrix $\mat{A}$ as the smallest number  $\delta_s=\delta_s(\mat{A})$  such that \cite{CanTao:05}
\begin{equation}
\label{eq:csric}
(1-\delta_s) ||\vect{x}||^2_2 \leq ||\vect{A x}||^2_2 \leq (1+ \delta_s) ||\vect{x}||^2_2
\end{equation}
holds for all $s$-sparse vectors $\vect{x}$. The possibility to use $\ell_1$ minimization instead of the impractical $\ell_0$ minimization is, for a given matrix $\mat{A}$, related to the restricted isometry constant \cite{CanTao:05}. For example, in \cite{CaiWanXu:10} it is shown that the $\ell_0$ and the $\ell_1$ solutions are coincident for $s$-sparse vectors $\vect{x}$ if $\delta_s < 0.307$. 

Then, the next question is how to design a matrix $\mat{A}$ with a prescribed isometry constant. One possible way to design $\mat{A}$ consists simply in randomly generating its entries according to some statistical distribution. 
 The target here is to find a way to generate $\mat{A}$ such that, for example, for given $m, n, s$, the probability $\Pr\left\{\delta_s(\mat{A}) < 0.307 \right\}$ is close to one. When the measurement matrix $\mat{A}$ has entries randomly generated according to a $\mathcal{N}(0,1/m)$ distribution, this probability can be bounded starting from the probability $\Pr\left\{a\leq \lambda_{\min}({\bf A}_s^T {\bf A}_s) , \lambda_{\max}({\bf A}_s^T {\bf A}_s)\leq b\right\}$, where 
 ${\bf A}_s$ is a $m \times s$ Gaussian random matrix with $\mathcal{N}(0,1/m)$ i.i.d. entries \cite[Sec. III]{CanTao:05}.  
In  \cite{CanTao:05}, deviation bounds for the largest and smallest eigenvalues of ${\bf A}_s^T {\bf A}_s$ are obtained, using the concentration inequality, as
\begin{align*}
\Pr\left\{\sqrt{\lambda_{\max}({\bf A}_s^T {\bf A}_s)} > 1+\sqrt{\frac{s}{m}}+ o(1)+t \right\} & \leq  e^{-m t^2 /2}  \\
\Pr\left\{\sqrt{\lambda_{\min}({\bf A}_s^T {\bf A}_s)} < 1-\sqrt{\frac{s}{m}}+ o(1)-t \right\} & \leq  e^{-m t^2 /2} 
\end{align*}
where $t>0$ and $o(1)$ is a small term tending to zero as $m$ increases. 
In our notation, and neglecting $o(1)$, the previous bounds can be rewritten
\begin{align}
\Pr\left\{\lambda_{\max}(\mat{M}) > \left(\sqrt{m}+\sqrt{s}+t \sqrt{m} \right)^2\right\} & \leq  e^{-m t^2 /2} \label{eq:csconine1}\\
\Pr\left\{\lambda_{\min}(\mat{M}) < \left(\sqrt{m}-\sqrt{s}- t \sqrt{m} \right)^2\right\} & \leq  e^{-m t^2 /2} \label{eq:csconine2}
\end{align}
where 
$\mat{M}=m {\bf A}_s^T {\bf A}_s \sim {\bf \mathcal{W}}_{s}(m, {\bf I})$ and $\lambda_{\max}(\mat{M}) =m \lambda_{\max}({\bf A}_s^T {\bf A}_s)$.  
In Fig.~\ref{fig:concineq} and Fig.~\ref{fig:concineq2} these bounds are compared with the exact results given by Algorithm~\ref{alg:wis} and with the simple gamma approximations \eqref{eq:lmaxTWMC}, \eqref{eq:lminTWMC}, for some values on $s, m$. It can be noted that the concentration inequality bounds \eqref{eq:csconine1}, \eqref{eq:csconine2} are quite loose. 
For example, from Fig. 1 we observe that at $t=0.15$ the new results (the two lower curves) are many orders of magnitude lower than the concentration bound (solid line). 
\begin{figure}[h]
\psfrag{TW1}{W_1}
\centerline{\includegraphics[width=1\columnwidth,draft=false]{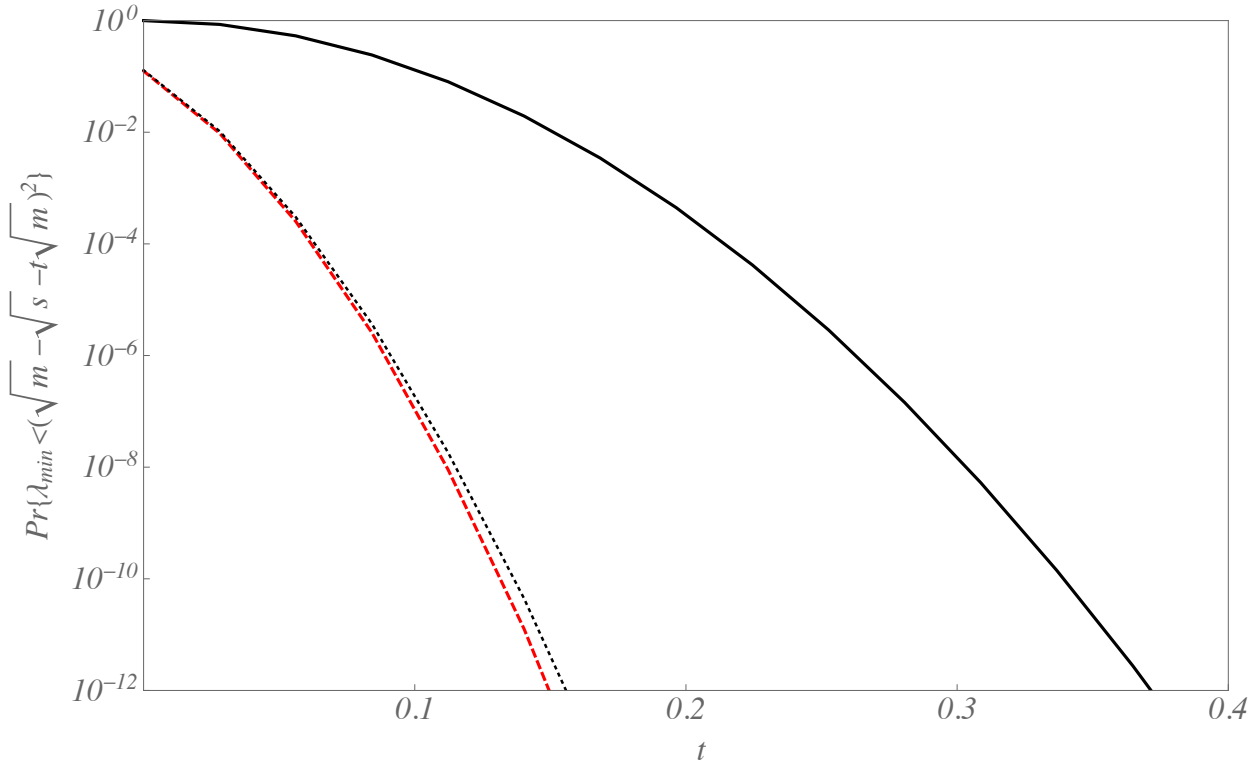}}
\caption{Distribution of the smallest eigenvalue for Wishart real matrices ${\mathcal{W}}_{s}(m, {\bf I})$, $m=400, s=10$. Comparison between the concentration inequality bound \eqref{eq:csconine2} (solid),  the gamma approximation \eqref{eq:lminTWMC} (dotted), and the exact (Alg. 1, dashed line).} \label{fig:concineq}
\end{figure}
\begin{figure}[h]
\psfrag{TW1}{W_1}
\centerline{\includegraphics[width=1\columnwidth,draft=false]{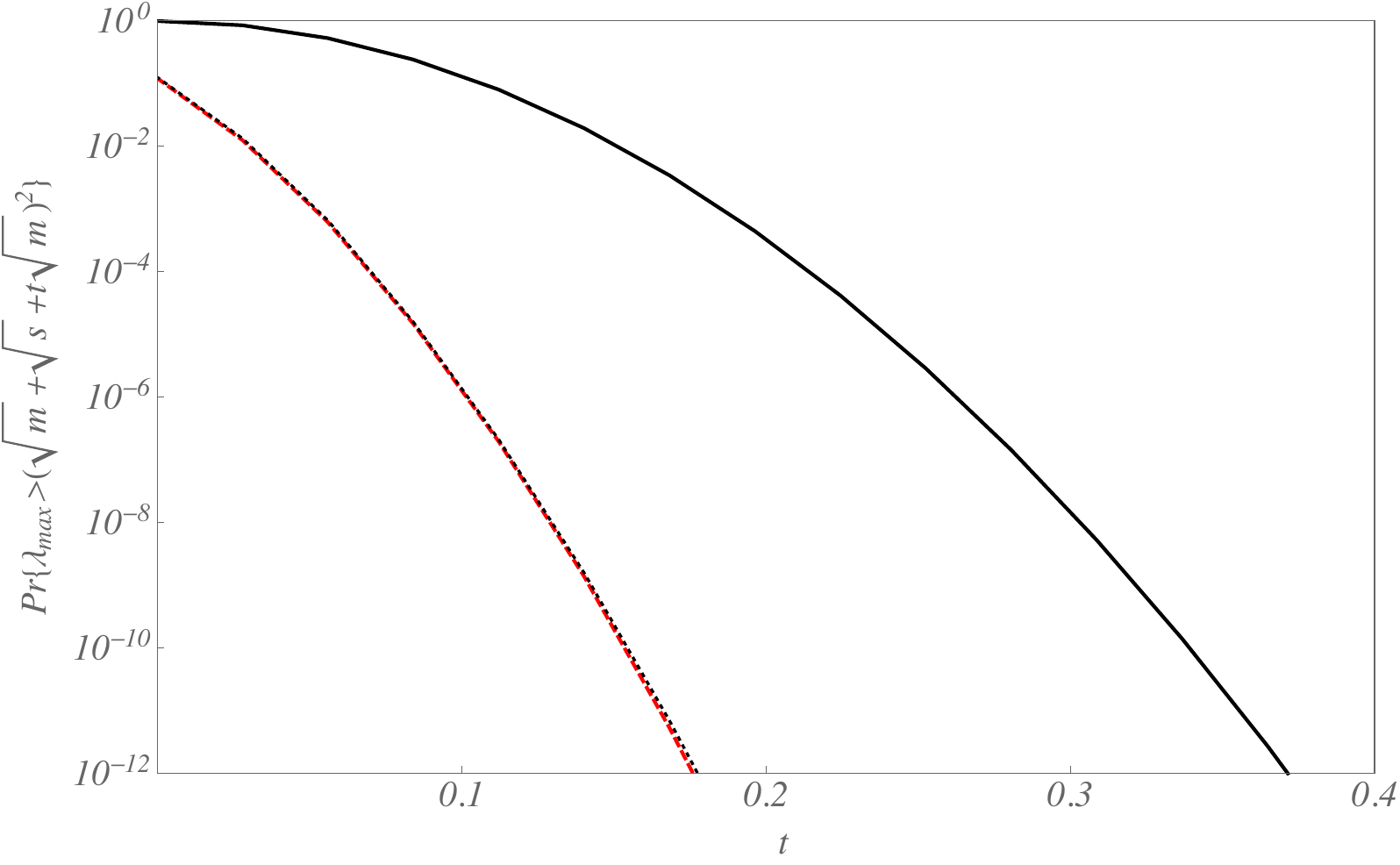}}
\caption{Distribution of the largest eigenvalue for Wishart real matrices ${\mathcal{W}}_{s}(m, {\bf I})$, $m=400, s=10$. Comparison between the concentration inequality bound \eqref{eq:csconine1} (solid),  the gamma approximation \eqref{eq:lmaxTWMC} (dotted), and the exact (Alg. 1, dashed line). The last two curves are very close.} \label{fig:concineq2}
\end{figure}


\section{Conclusions}
Iterative algorithms have been found to evaluate in few seconds the exact value of the probability that all eigenvalues lie within an arbitrary interval $[a,b]$, for quite large (e.g. $500 \times 500$) real white Wishart, complex Wishart with arbitrary correlation, double Wishart, and Gaussian symmetric/Hermitian matrices. 
These exact results for finite dimensions are therefore complementary to  methods for the analysis of asymptotically large matrices, like the approaches based on Coulomb gas models \cite{DeaMaj:06,MajViv:12}. 

%
Simple approximations based on shifted incomplete gamma functions have also been proposed, and it is proved that for increasingly large matrices the probability that all eigenvalues are within the limiting support is $0.6921$ for real white Wishart and GOE, and $0.9397$ for complex white Wishart and GUE.

For instance, we analyzed the probability that all eigenvalues are negative for GOE, of interest in complex ecosystems and physics. As another example,  in the context of compressed sensing we compared the new expressions with the concentration inequality based bounds.

\section*{Acknowledgements}
The author would like to thank the Reviewers for constructive comments, and A. Elzanaty, A. Giorgetti and A. Mariani for suggestions and discussions. 

\bibliographystyle{IEEEtran}
\bibliography{IEEEabrv,BiblioMCCV,RandomMatrix,MyBooks,MIMO,BibEIGEN}

\begin{thebibliography}{10}
\providecommand{\url}[1]{#1}
\csname url@samestyle\endcsname
\providecommand{\newblock}{\relax}
\providecommand{\bibinfo}[2]{#2}
\providecommand{\BIBentrySTDinterwordspacing}{\spaceskip=0pt\relax}
\providecommand{\BIBentryALTinterwordstretchfactor}{4}
\providecommand{\BIBentryALTinterwordspacing}{\spaceskip=\fontdimen2\font plus
\BIBentryALTinterwordstretchfactor\fontdimen3\font minus
  \fontdimen4\font\relax}
\providecommand{\BIBforeignlanguage}[2]{{%
\expandafter\ifx\csname l@#1\endcsname\relax
\typeout{** WARNING: IEEEtran.bst: No hyphenation pattern has been}%
\typeout{** loaded for the language `#1'. Using the pattern for}%
\typeout{** the default language instead.}%
\else
\language=\csname l@#1\endcsname
\fi
#2}}
\providecommand{\BIBdecl}{\relax}
\BIBdecl

\bibitem{And:03}
T.~W. Anderson, \emph{An Introduction to Multivariate Statistical
  Analysis}.\hskip 1em plus 0.5em minus 0.4em\relax New York: Wiley, 2003.

\bibitem{Mui:B82}
R.~J. Muirhead, \emph{Aspects of Multivariate Statistical Theory}.\hskip 1em
  plus 0.5em minus 0.4em\relax New York: Wiley, 1982.

\bibitem{Meh:91}
M.~L. Mehta, \emph{Random Matrices}, 2nd~ed.\hskip 1em plus 0.5em minus
  0.4em\relax Boston, MA: Academic, 1991.

\bibitem{Wint:87}
J.~H. Winters, ``On the capacity of radio communication systems with diversity
  in {R}ayleigh fading environment,'' \emph{{IEEE} J. Sel. Areas Commun.},
  vol.~5, no.~5, pp. 871--878, Jun. 1987.

\bibitem{Ede:88}
A.~Edelman, ``Eigenvalues and condition numbers of random matrices,''
  \emph{SIAM Journal on Matrix Analysis and Applications}, vol. 1988, pp.
  543--560, 1988.

\bibitem{Tel:99}
{\.I}.~E. Telatar, ``Capacity of multi-antenna {G}aussian channels,''
  \emph{European Trans. Telecommun.}, vol.~10, no.~6, pp. 585--595, Nov./Dec.
  1999.

\bibitem{Joh:01}
I.~Johnstone, ``On the distribution of the largest eigenvalue in principal
  components analysis,'' \emph{The Annals of Statistics}, vol.~29, no.~2, pp.
  295--327, 2001.

\bibitem{ChiWinZan:J03}
M.~Chiani, M.~Z. Win, and A.~Zanella, ``On the capacity of spatially correlated
  {MIMO} {Rayleigh} fading channels,'' \emph{{IEEE} Trans. Inf. Theory},
  vol.~49, no.~10, pp. 2363--2371, Oct. 2003.

\bibitem{CanTao:05}
E.~J. Cand{\`e}s and T.~Tao, ``Decoding by linear programming,'' \emph{{IEEE}
  Trans. Inf. Theory}, vol.~51, no.~12, pp. 4203--4215, Dec 2005.

\bibitem{PenGar:09}
F.~Penna, R.~Garello, and M.~A. Spirito, ``Cooperative spectrum sensing based
  on the limiting eigenvalue ratio distribution in {Wishart} matrices,''
  \emph{IEEE Commun. Lett.}, vol.~13, no.~7, pp. 507--509, 2009.

\bibitem{CheMcK:12}
Y.~Chen and M.~McKay, ``Coulumb fluid, {Painlev\'e} transcendents, and the
  information theory of {MIMO} systems,'' \emph{{IEEE} Trans. Inf. Theory},
  vol.~58, no.~7, pp. 4594--4634, July 2012.

\bibitem{Don:06}
D.~L. Donoho, ``Compressed sensing,'' \emph{{IEEE} Trans. Inf. Theory},
  vol.~52, no.~4, pp. 1289--1306, 2006.

\bibitem{Can:08}
E.~J. Cand{\`e}s and M.~B. Wakin, ``An introduction to compressive sampling,''
  \emph{Signal Processing Magazine, IEEE}, vol.~25, no.~2, pp. 21--30, 2008.

\bibitem{May:72}
R.~M. {May}, ``{Will a Large Complex System be Stable?}'' \emph{Nature}, vol.
  238, pp. 413--414, Aug. 1972.

\bibitem{AazEas:06}
A.~Aazami and R.~Easther, ``{Cosmology from random multifield potentials},''
  \emph{Journal of Cosmology and Astroparticle Physics}, vol. 0603, p. 013,
  2006.

\bibitem{DeaMaj:08}
D.~S. {Dean} and S.~N. {Majumdar}, ``{Extreme value statistics of eigenvalues
  of Gaussian random matrices},'' \emph{Physical Review E}, vol.~77, no.~4, p.
  041108, Apr. 2008.

\bibitem{MarMca:13}
M.~C.~D. {Marsh}, L.~{McAllister}, E.~{Pajer}, and T.~{Wrase}, ``{Charting an
  Inflationary Landscape with Random Matrix Theory},'' \emph{Journal of
  Cosmology and Astroparticle Physics}, vol.~11, p.~40, Nov. 2013.

\bibitem{DedMal:07}
J.-P. {Dedieu} and G.~{Malajovich}, ``{On the number of minima of a random
  polynomial},'' \emph{ArXiv Mathematics e-prints}, Feb. 2007.

\bibitem{BaiSil:06}
Z.~Bai and J.~W. Silverstein, \emph{Spectral Analysis of Large Dimensional
  Random Matrices}.\hskip 1em plus 0.5em minus 0.4em\relax Science Press, 2006,
  2006.

\bibitem{TraWid:09}
C.~Tracy and H.~Widom, ``The distributions of random matrix theory and their
  applications,'' \emph{New Trends in Mathematical Physics}, pp. 753--765,
  2009.

\bibitem{TraWid:94}
------, ``{Level-spacing distributions and the Airy kernel},''
  \emph{Communications in Mathematical Physics}, vol. 159, no.~1, pp. 151--174,
  1994.

\bibitem{TraWid:96}
------, ``On orthogonal and symplectic matrix ensembles,'' \emph{Communications
  in Mathematical Physics}, vol. 177, pp. 727--754, 1996.

\bibitem{Joha:00}
K.~Johansson, ``Shape fluctuations and random matrices,'' \emph{Communications
  in Mathematical Physics}, vol. 209, pp. 437--476, 2000.

\bibitem{Sos:02}
A.~Soshnikov, ``A note on universality of the distribution of the largest
  eigenvalues in certain sample covariance matrices,'' \emph{Journal of
  Statistical Physics}, vol. 108, pp. 1033--1056, 2002.

\bibitem{Joh:09}
I.~M. Johnstone, ``Approximate null distribution of the largest root in
  multivariate analysis,'' \emph{The annals of Applied Statistics}, vol.~3,
  no.~4, pp. 1616--1633, 2009.

\bibitem{FelSod:10}
O.~N. Feldheim and S.~Sodin, ``A universality result for the smallest
  eigenvalues of certain sample covariance matrices,'' \emph{Geometric And
  Functional Analysis}, vol.~20, no.~1, pp. 88--123, 2010.

\bibitem{Ma:12}
Z.~Ma, ``Accuracy of the {Tracy--Widom limits for the extreme eigenvalues in
  white Wishart matrices},'' \emph{Bernoulli}, vol.~18, no.~1, pp. 322--359,
  2012.

\bibitem{DeaMaj:06}
D.~S. Dean and S.~N. Majumdar, ``Large deviations of extreme eigenvalues of
  random matrices,'' \emph{Phys. Rev. Lett.}, vol.~97, p. 160201, Oct 2006.

\bibitem{NadMaj:11}
C.~Nadal and S.~N. Majumdar, ``{A simple derivation of the Tracy-Widom
  distribution of the maximal eigenvalue of a Gaussian unitary random
  matrix},'' \emph{Journal of Statistical Mechanics: Theory and Experiment},
  vol. 2011, no.~04, p. P04001, 2011.

\bibitem{Chi:J14}
M.~Chiani, ``{Distribution of the largest eigenvalue for real Wishart and
  Gaussian random matrices and a simple approximation for the Tracy-Widom
  distribution},'' \emph{Journal of Multivariate Analysis}, vol. 129, pp. 69 --
  81, 2014.

\bibitem{Chi:J16}
------, ``{Distribution of the largest root of a matrix for Roy's test in
  multivariate analysis of variance},'' \emph{Journal of Multivariate
  Analysis}, vol. 143, pp. 467--471, 2016, also in arxiv, 2014.

\bibitem{ChiZan:C08}
M.~Chiani and A.~Zanella, ``Joint distribution of an arbitrary subset of the
  ordered eigenvalues of {Wishart} matrices,'' in \emph{Proc. IEEE Int. Symp.
  on Personal, Indoor and Mobile Radio Commun.}, Cannes, France, Sep. 2008, pp.
  1--6.

\bibitem{AbrSte:B70}
M.~Abramowitz and I.~A. Stegun, \emph{Handbook of Mathematical Functions wih
  Formulas, Graphs, and Mathematical Tables}.\hskip 1em plus 0.5em minus
  0.4em\relax Washington, D.C.: United States Department of Commerce, 1970.

\bibitem{Jam:64}
A.~T. James, ``Distributions of matrix variates and latent roots derived from
  normal samples,'' \emph{Annals Math. Stat.}, vol.~35, pp. 475--501, 1964.

\bibitem{Deb:55}
N.~De~Bruijn, ``On some multiple integrals involving determinants,'' \emph{J.
  Indian Math. Soc}, vol.~19, pp. 133--151, 1955.

\bibitem{VivMaj:07}
P.~Vivo, S.~N. Majumdar, and O.~Bohigas, ``{Large deviations of the maximum
  eigenvalue in Wishart random matrices},'' \emph{Journal of Physics A:
  Mathematical and Theoretical}, vol.~40, no.~16, p. 4317, 2007.

\bibitem{MajViv:12}
S.~N. Majumdar and P.~Vivo, ``Number of relevant directions in principal
  component analysis and wishart random matrices,'' \emph{Phys. Rev. Lett.},
  vol. 108, p. 200601, May 2012.

\bibitem{Pil:67}
K.~S. Pillai, ``Upper percentage points of the largest root of a matrix in
  multivariate analysis,'' \emph{Biometrika}, vol.~54, no. 1-2, pp. 189--194,
  1967.

\bibitem{Kha:64}
C.~G. Khatri, ``Distribution of the largest or the smallest characteristic root
  under null hypothesis concerning complex multivariate normal populations,''
  \emph{Ann. Math. Stat.}, vol.~35, pp. 1807--1810, Dec. 1964.

\bibitem{ChiWinShi:J10}
M.~Chiani, M.~Z. Win, and H.~Shin, ``{MIMO} networks: the effects of
  interference,'' \emph{{IEEE} Trans. Inf. Theory}, vol.~56, no.~1, pp.
  336--349, Jan. 2010.

\bibitem{ZanChiWin:J09}
A.~Zanella, M.~Chiani, and M.~Z. Win, ``On the marginal distribution of the
  eigenvalues of {W}ishart matrices,'' \emph{IEEE Trans. Commun.}, vol.~57,
  no.~4, pp. 1050--1060, Apr. 2009.

\bibitem{Nad:08}
B.~Nadler, ``Finite sample approximation results for principal component
  analysis: a matrix perturbation approach,'' \emph{The Annals of Statistics},
  vol.~36, no.~6, pp. pp. 2791--2817, 2008.

\bibitem{Elk:03}
N.~El~Karoui, ``On the largest eigenvalue of {Wishart} matrices with identity
  covariance when $n,p$ and $p/n \rightarrow \infty$,'' \emph{arXiv preprint
  math/0309355}, 2003.

\bibitem{Pec:09}
S.~P{\'e}ch{\'e}, ``Universality results for the largest eigenvalues of some
  sample covariance matrix ensembles,'' \emph{Probability Theory and Related
  Fields}, vol. 143, no.~3, pp. 481--516, 2009.

\bibitem{MarPas:67}
V.~A. Mar{\v{c}}enko and L.~A. Pastur, ``Distribution of eigenvalues for some
  sets of random matrices,'' \emph{Math USSR Sbornik}, vol.~1, pp. 457--483,
  1967.

\bibitem{CaiWanXu:10}
T.~Cai, L.~Wang, and G.~Xu, ``New bounds for restricted isometry constants,''
  \emph{{IEEE} Trans. Inf. Theory}, vol.~56, no.~9, pp. 4388--4394, Sept 2010.

\end{thebibliography}

\end{document}